\documentclass[12pt]{amsart}

\usepackage[margin=1.25in]{geometry}
\usepackage{amsmath, amsfonts, amssymb, amscd, amsthm, mathtools, stmaryrd, mathrsfs, graphicx}
\usepackage{amsxtra}
\usepackage[abbrev,alphabetic]{amsrefs}
\usepackage{hyperref}
\usepackage{cleveref}
\usepackage{comment}
\usepackage[all,cmtip]{xy}
\usepackage{tikz-cd}
\usetikzlibrary{cd}
\usepackage{booktabs}
\usepackage{multirow}
\usepackage{xcolor, soul}
\setstcolor{red}

\tikzcdset{
  cells={font=\everymath\expandafter{\the\everymath\displaystyle}},
}

\newtagform{tiny}{\tiny(}{)}

\makeatletter
\def\@tocline#1#2#3#4#5#6#7{\relax
  \ifnum #1>\c@tocdepth \else
    \par \addpenalty\@secpenalty\addvspace{#2}%
    \begingroup \hyphenpenalty\@M
    \@tempdima\if@empty{#4}{\csname r@tocindent\number#1\endcsname}{#4}\relax
    \parindent\z@ \leftskip#3\relax \advance\leftskip\@tempdima\relax
    \rightskip\@pnumwidth plus4em \parfillskip-\@pnumwidth
    #5\leavevmode\hskip-\@tempdima
    \ifcase #1 \or\or \hskip 1em \or \hskip 2em \else \hskip 3em \fi%
    #6\nobreak\hfill\hbox to\@pnumwidth{\@tocpagenum{#7}}\par
    \nobreak
    \endgroup
  \fi}
\makeatother



\title{Fedder-type criterion for quasi-$F^e$-splitting and quasi-$F$-regularity}
\author{Shou Yoshikawa}
\address{Institute of Science Tokyo, Tokyo 152-8551, Japan}
\email{yoshikawa.s.9fe9@m.isct.ac.jp}

\newcommand{\Z}{\mathbb{Z}} \newcommand{\Q}{\mathbb{Q}} \newcommand{\R}{\mathbb{R}} \newcommand{\F}{\mathbb{F}}

  \newcommand{\cO}{\mathcal{O}}

\newcommand{\m}{\mathfrak{m}} \newcommand{\n}{\mathfrak{n}}  

\renewcommand{\mod}{\ \textrm{mod}\ }

\DeclareMathOperator{\Spec}{Spec}

\DeclareMathOperator{\Hom}{Hom}
\DeclareMathOperator{\Ext}{Ext}

\DeclareMathOperator{\Ker}{Ker}

\DeclareMathOperator{\fpt}{fpt}
\DeclareMathOperator{\bfpt}{bfpt}
\DeclareMathOperator{\qfpt}{qfpt}

\newcommand{\wt}{\widetilde}

\newcommand{\cond}[1]{\textup{(#1)}}
\newcommand{\paren}[1]{\left( #1 \right)}
\newcommand{\eqtag}[1]{\overset{(\star_{#1})}{=}}
\newcommand{\rdown}[1]{\lfloor #1 \rfloor}

\theoremstyle{plain}
\newtheorem{theorem}{Theorem}[section]

\newtheorem{proposition}[theorem]{Proposition}
\newtheorem{lemma}[theorem]{Lemma}
\newtheorem{corollary}[theorem]{Corollary}

\newtheorem{claim}[theorem]{Claim}

\newtheorem*{claim*}{Claim}
\newtheorem{theoremA}{Theorem}

\newtheorem{corollaryA}[theoremA]{Corollary}

\theoremstyle{definition}
\newtheorem{definition}[theorem]{Definition}

\newtheorem{example}[theorem]{Example}
\newtheorem{notation}[theorem]{Notation}
\newtheorem*{setup*}{Setup}

\theoremstyle{remark}
\newtheorem{remark}[theorem]{Remark}
\newtheorem*{ackn}{Acknowledgements}

\numberwithin{equation}{section}

\crefname{theorem}{Theorem}{Theorems}
\crefname{proposition}{Proposition}{Propositions}
\crefname{lemma}{Lemma}{Lemmas}
\crefname{corollary}{Corollary}{Corollaries}
\crefname{conjecture}{Conjecture}{Conjectures}
\crefname{claim}{Claim}{Claims}
\crefname{notation}{Notation}{Notations}
\crefname{theoremA}{Theorem}{Theorems}
\crefname{corollaryA}{Corollary}{Corollaries}
\crefname{example}{Example}{Examples}

\newenvironment{claimproof}[0]
  {%
   \paragraph{\it Proof.}%
  }
  {%
    \hfill$\blacksquare$%
  }

\numberwithin{equation}{section}

\begin{document}

\begin{abstract}
We study quasi-$F^e$-split and quasi-$F$-regular singularities, which generalize Yobuko’s quasi-$F$-splitting. We establish Fedder-type criteria that characterize these properties for hypersurfaces. These criteria offer explicit tools for computation and verification. As an application, we construct a counterexample to the inversion of adjunction for quasi-$F$-regularity and compute the quasi-$F$-pure threshold of the cone over the ordinary cusp.
\end{abstract}

\maketitle

\setcounter{tocdepth}{1}

\section{Introduction}

Quasi-$F$-splitting, introduced by Yobuko \cite{Yobuko19}, is a class of singularities in positive characteristic with various applications in algebraic geometry. In recent years, it has also been studied in connection with birational geometry (see \cite{KTTWYY1}, \cite{KTTWYY2}). A Fedder-type criterion, proved in \cite{KTY}, provides a method for determining when a hypersurface is quasi-$F$-split, and many concrete examples have been constructed using this approach (see \cite{KTY2}).
To further develop the theory of quasi-$F$-splitting and its generalizations, the notions of \emph{quasi-$F^e$-splitting} and \emph{quasi-$F$-regularity} were introduced in \cite{TWY24}, and their foundational properties were investigated in \cite{KTTWYY3}.

In this paper, we establish Fedder-type criteria for quasi-$F^e$-splitting and quasi-$F$-regularity. These results provide practical tools for verifying these properties in explicit examples. We begin by recalling the definitions of quasi-$F^e$-splitting and quasi-$F$-regularity.

\begin{definition}[\cite{TWY24}*{Sections~3.1 and 3.5}]
Let $S$ be an $F$-finite Gorenstein local $\F_p$-algebra, and let $e, n \geq 1$ be integers.
\begin{enumerate}
    \item For $c \in S$, define a $W_n(S)$-module $Q^e_{S,c,n}$ by the following pushout diagram: 
    \[
    \begin{tikzcd}
    W_n(S) \ar[r] \ar[d] & F^e_*W_n(S) \arrow[r,"\cdot F^e_*\text{[}c\text{]}"] & F^e_*W_n(S) \arrow[d] \\
    S \arrow[rr] & & Q^e_{S,c,n}.
    \end{tikzcd}
    \]
    When $c = 1$, we write $Q^e_{S,n} := Q^e_{S,1,n}$. The induced maps $S \to Q^e_{S,c,n}$ and $S \to Q^e_{S,n}$ are denoted by $\Phi^e_{S,c,n}$ and $\Phi^e_{S,n}$, respectively.
    
    \item We say that $S$ is \emph{$n$-quasi-$F^e$-split} if the map
    \[
    (\Phi^e_{S,n})^* \colon \Hom_{W_n(S)}(Q^e_{S,n},W_n\omega_S) \to \Hom_{W_n(S)}(S,W_n\omega_S)
    \]
    is surjective, where $W_n\omega_S$ is a dualizing sheaf of $W_n(S)$ (see \cite{KTTWYY3}*{Section~9}).
    
    \item We say that $S$ is \emph{$n$-quasi-$F$-regular} if, for every non-zero divisor $c \in S$, there exists $e_0 \geq 1$ such that the map
    \[
    (\Phi^{e'}_{S,c,n})^* \colon \Hom_{W_n(S)}(Q^{e'}_{S,c,n},W_n\omega_S) \to \Hom_{W_n(S)}(S,W_n\omega_S)
    \]
    is surjective for all $e' \geq e_0$.
\end{enumerate}
\end{definition}
\noindent
To state our main theorems, we first introduce some notation.
Let $k$ be a perfect field, $A:=W(k)[[x_1,\ldots,x_N]]$, $R:=A/pA$, and $\m:=(x_1,\ldots,x_N)$.
We set $\m^{[p]}:=(x_1^p,\ldots,x_N^p)$.
Let $\phi \colon A \to A$ be the Frobenius lift such that $\phi|_{W(k)}=F$ and $\phi(x_i)=x_i^p$.
Let $u$ be a generator of $\Hom_A(\phi_* A, A)$ as a $\phi_*A$-module.

Our first main result gives a Fedder-type criterion for quasi-$F^e$-splitting:

\begin{theoremA}[cf.~\cref{thm:structure of sigma'}]\label{intro:Fedder-qf^es}
Let $n, e \geq 1$ be integers and $f \in A$ such that $f,p$ is a regular sequence. Then $R/fR$ is $n$-quasi-$F^e$-split if and only if there exists $g \in A$ such that:
\begin{itemize}
    \item[(D1)] $u^{e+r-1}(\phi^{e+r-1}_*g) \in (p^r)$ for $1 \leq r \leq n-1$,
    \item[(D2)] $g$ admits a decomposition
    \[
    g \equiv g_0 + p g_1 + \cdots + p^{n-1} g_{n-1} \pmod{p^n}
    \]
    such that $u^r(\phi^r_*g_r) \in (f^{p^{e+n-r-1} - 1},p^n)$ for $0 \leq r \leq n-1$,
    \item[(D3)] $u^{e+n-1}(\phi_*^{e+n-1}g) \notin (\mathfrak{m}, p^n)$.
\end{itemize}
\end{theoremA}
\noindent
As a consequence, we obtain the following necessary and sufficient conditions, which are useful for explicit computations:

\begin{corollaryA}[cf.~\cref{nec-suff-cond}]\label{intro:nec-suff-cond}
Let $n, e \geq 1$ be integers and $f \in A$  such that $p,f$ is a regular sequence. 
\begin{enumerate}
    \item If there exists $g \in f^{p^{e+n-1}-1} A$ satisfying conditions \textup{(D1)} and \textup{(D3)}, then $R/fR$ is $n$-quasi-$F^e$-split.

    \item Define a sequence of ideals $\{I^e_n\}$ of $A$ inductively as follows. Set
    \[
    I^e_1 := f^{p-1} u^{e-1}(\phi_*^{e-1}f^{p^{e-1} - 1} A),
    \]
    and for each $n \geq 1$,
    \[
    I^e_{n+1} := u\left(\phi_*^{}\left( \Delta_1(f^{p-1}) ( I^e_n \cap u^{-1}(pA) )\right) \right) + f^{p-1}A.
    \]
    If $R/fR$ is $n$-quasi-$F^e$-split, then $I^e_n \not\subseteq (\mathfrak{m}^{[p]},p)$, where $\Delta_1$ is defined in \cite{KTY}*{Section~3}.

    \item Define another sequence of ideals $\{I'_n\}$ of $A$ as follows. Set
    \[
    I'_1 := f^{p-1} \mathfrak{m} ,
    \]
    and for each $n \geq 1$,
    \[
    I'_{n+1} := u\left(\phi_*^{}\left( \Delta_1(f^{p-1}) ( I'_n \cap u^{-1}(pA) )\right) \right) + f^{p-1}A.
    \]
    If $R/fR$ is not $F$-pure and $I'_n \subseteq (\mathfrak{m}^{[p]},p)$, then $R$ is not $n$-quasi-$F^2$-split.
\end{enumerate}
\end{corollaryA}

Our second main result provides a Fedder-type criterion for quasi-$F$-regularity:

\begin{theoremA}[\cref{fedder-qFr}]\label{intro:Fedder-qFr}
Let $n, e \geq 1$ be integers, $f,c \in A$, and $t \in \tau(R/fR)$, such that $f,p$ is a regular sequence and
\[
(A \to R/fR)(c) \in (t^4) \cap (R/fR)^\circ.
\]
Then $R/fR$ is $n$-quasi-$F$-regular if and only if there exists $g \in A$ such that $g$ satisfies condition \textup{(D2)} and $g c^{p^n - 1}$ satisfies conditions \textup{(D1)} and \textup{(D3)} in \cref{intro:Fedder-qf^es}.
\end{theoremA}
\noindent
As an application, we construct a counterexample to inversion of adjunction for quasi-$F$-regularity:

\begin{example}[\cref{example}~(2)]
Let $A := \mathbb{Z}_{(p)}[[x, y, z, w, v, s]]$, $R := A/p$, $f := x y s^2 + z w v^2 + y^3 w + x^3 z$, and $p = 2$.  
Then $R/fR$ is not quasi-$F$-split, while $R/(f, s)$ is $2$-quasi-$F$-regular.
\end{example}

\noindent
Furthermore, using the above criteria, the author and Takamatsu compute the quasi-$F^\infty$-height and the quasi-$F$-regular height explicitly in \cite{TY26}.

In the final part of this paper (Section~6), we apply our Fedder-type criteria to the study of quasi-$F$-pure thresholds (cf.~\Cref{defn:threshold}). 
In particular, we establish a criterion (Theorem~6.9) for determining whether a pair $(X,D)$ is quasi-$F^e$-split, which plays a crucial role in explicit computations of thresholds. 
As a concrete application, we compute the quasi-$F$-pure threshold of the cone over the ordinary cusp (Example~6.10), which can be stated as follows:
\begin{theoremA}[\cref{cusp}]\label{intro:thm:cusp}
Let $k$ be a perfect field.
Let $R=k[[x,y,z]]$ and $D=\mathrm{div}(x^3+y^2z)$.
Then $\qfpt(X;D)=5/6$.
\end{theoremA}

This example highlights a striking phenomenon: while the $F$-pure threshold of the cone depends on the characteristic of the base field (see \cite{MTW}*{Example~4.3}), the quasi-$F$-pure threshold takes a constant value $5/6$, independent of the characteristic. 
This illustrates the robustness of quasi-$F$-splittings compared to classical $F$-singularity thresholds, and demonstrates how our criterion provides effective tools for analyzing such phenomena.

\begin{ackn}
The author wishes to express his gratitude to Ryo Ishizuka, Teppei Takamatsu, and Tatsuro Kawakami for valuable discussions.
He is also grateful to Hiromu Tanaka for helpful comments. 
The author would like to thank the anonymous referee for careful reading and valuable suggestions, which helped to improve the presentation of this paper.
The author was supported by JSPS KAKENHI Grant number JP24K16889.
\end{ackn}

\section{Preliminary}

\subsection{Notation}\label{ss-notation}
In this subsection, we summarize the notation and  terminologies used in this paper. 

\begin{enumerate}
\item Throughout this paper, we fix a prime number $p$ and we set $\F_p := \Z/p\Z$. 
\item For a ring $R$, the set of non-zero divisors is denoted by $R^\circ$.
\item For a ring $R$ of characteristic $p>0$, $F: R\to R$ denotes the absolute Frobenius ring homomorphism, that is, $F(r)=r^p$ for every $r \in R$.  
\item For a ring $R$ of characteristic $p>0$, we say that $R$ is {\em $F$-finite} if $F$ is finite. 
\item For a ring $R$ and a ring homomorphism $\phi \colon R \to R$, we say that $\phi$ is a \emph{lift of Frobenius} if $\phi(a) \equiv a^p \mod p$ for every $a \in R$. 
\item For a ring $R$, \emph{ring of Witt vectors} over $R$ is denoted by $W(R)$.
Furthermore, \emph{Witt vectors over $R$ of length $n$} is denoted by $W_n(R)$ (see \cite{Serre}*{II, \S{6}} for details).
Moreover, the \emph{shift map} $W_{n}(R) \to W_{n+1}(R)$ is denoted by $V$ and the \emph{restriction map} $W_{n+1}(R) \to W_n(R)$ is denoted by $Res$ for every integer $n \geq 1$.
\item 
Let $R$ be an $F$-finite Noetherian $\F_p$-algebra. 
Then $W_n(R)$ has a dualizing module $W_n\omega_R$ such that 
\[
\Hom_{W_m(R)}(W_n(R),W_m\omega_R) \simeq W_n\omega_R
\]
for $m \geq n$ by \cite{KTTWYY3}*{Section~9}.
Furthermore, we define a $W_{n+1}(R)$-module homomorphism
\[
T_R \colon W_n\omega_R \longrightarrow W_{n+1}\omega_R
\]
to be the Grothendieck trace map associated with the restriction homomorphism 
\[
Res \colon W_{n+1}(R) \longrightarrow W_n(R).
\]
More precisely, via the natural identifications
\[
W_n\omega_R 
\simeq \Hom_{W_n(R)}(W_n(R), W_n\omega_R)
\simeq \Hom_{W_{n+1}(R)}(W_n(R), W_{n+1}\omega_R),
\]
the map $T_R$ is defined as the composite
\[
\Hom_{W_{n+1}(R)}(W_n(R), W_{n+1}\omega_R)
\xrightarrow{\ Res^*\ }
\Hom_{W_{n+1}(R)}(W_{n+1}(R), W_{n+1}\omega_R)
\simeq W_{n+1}\omega_R,
\]
where $Res^*$ denotes the dual morphism induced by 
\[
Res \colon W_{n+1}(R) \longrightarrow W_n(R).
\]
\item Let $R$ be an $F$-finite Noetherian $\F_p$-algebra. 
For a $W_n(R)$-module $M$, we define
\[
M^* :=\Hom_{W_n(R)}(M,W_n\omega_R).
\]
Furthermore, for a $W_n(R)$-module homomorphism $f \colon M \to N$ of $W_n(R)$-modules, we define
\[
f^* \colon N^* \to M^*.
\]
\end{enumerate}

\subsection{Quasi-$F^e$-splitting and quasi-$F$-regularity}

\begin{definition}\textup{\cite{TWY24}*{Section~3.1, 3.5}}
Let $(R,\m)$ be an $F$-finite Gorenstein local $\F_p$-algebra.
We fix an isomorphism $R \simeq \omega_R$.
Let $e,n \geq 1$ be integers.
\begin{enumerate}
    \item For $c \in R$, we define a $W_n(R)$-module $Q^e_{R,c,n}$ by the following pushout diagram:
    \[
    \begin{tikzcd}
    W_n(R) \ar[r] \ar[d] & F^e_*W_n(R) \arrow[r,"\cdot F^e_*\text{[}c\text{]}"] & F^e_*W_n(R) \arrow[d] \\
    R \arrow[rr] & & Q^e_{R,c,n}.
    \end{tikzcd}
    \]
    If $c=1$, the module $Q^e_{R,c,n}$ is denoted by $Q^e_{R,n}$.
    Furthermore, the homomorphism $R \to Q^e_{R,c,n}$ and $R \to Q^e_{R,n}$ induced by the pushout diagram are denoted by $\Phi^e_{R,c,n}$ and $\Phi^e_{R,n}$, respectively.
    \item We say that $R$ is \emph{$n$-quasi-$F^e$-split} if the homomorphism
    \[
    (\Phi^e_{R,n})^* \colon \Hom_{W_n(R)}(Q^e_{R,n},W_n\omega_R) \to \Hom_{W_n(R)}(R,W_n\omega_R)
    \]
    is surjective.
    \item We say that $R$ is \emph{$n$-quasi-$F$-regular} if for every $c \in R^\circ$, there exists $e_0 \geq 1$ such that the homomorphism
    \[
    (\Phi^{e'}_{R,c,n})^* \colon \Hom_{W_n(R)}(Q^{e'}_{R,c,n},W_n\omega_R) \to \Hom_{W_n(R)}(R,W_n\omega_R)
    \]
    is surjective for every $e' \geq e_0$.
    \item For $f,c \in R$, we define a $W_n(R)$-module $Q^{f,e}_{R,c,n}$ by the following pushout diagram:
    \[
    \begin{tikzcd}
    W_n(fR) \ar[r] \ar[d] & F^e_*W_n(fR) \arrow[r,"\cdot F^e_*\text{[}c\text{]}"] & F^e_*W_n(fR) \arrow[d] \\
    fR \arrow[rr] & & Q^{f,e}_{R,c,n},
    \end{tikzcd}
    \]
    where 
    \[
W_n(fR):=\{(fa_0,\ldots,fa_{n-1}) \mid a_0,\ldots,a_{n-1} \in R \}.
    \]
    In other words, we have
    \[
W_n(fR)=\Ker(W_n(R) \to W_n(R/fR)).
    \]
    If $c=1$, the module $Q^{f,e}_{R,c,n}$ is denoted by $Q^{f,e}_{R,n}$.
    Furthermore, the homomorphism $fR \to Q^{f,e}_{R,c,n}$ and $R \to Q^{f,e}_{R,n}$ induced by the pushout diagram are denoted by $\Phi^{f,e}_{R,c,n}$ and $\Phi^{f,e}_{R,n}$, respectively.
\end{enumerate}
\end{definition}

\begin{proposition}\label{CM}
Let $(R,\m)$ be an $F$-finite Gorenstein local $\F_p$-algebra.
If $R$ is $F$-pure, then $Q^e_{R,c,n}$ is Cohen-Macaulay for an element $c \in R^\circ$ and integers $e,n \geq 1$.
\end{proposition}

\begin{proof}
We take integers $e,n \geq 1$ and $c \in R^\circ$.
We prove the assertion by induction on $n$.
For $n=1$, we have $Q^e_{R,c,1} \simeq F^e_*R$, thus it is Cohen-Macaulay.
Next, we assume $n \geq 2$.
Since we have the following commutative diagram whose rows are exact,
\[
\begin{tikzcd}
0 \arrow[r]
  & F^{n-1}_*R \arrow[r,"V^{n-2}"] \arrow[d,"F^e"]
  & F_*W_{n-1}(R) \arrow[r,"Res"] \arrow[d,"{\cdot\,F^e_*\!\left[c\right]\circ F^eV}"]
  & F_*W_{n-2}(R) \arrow[r] \arrow[d,"{\cdot\,F^e_*\!\left[c\right]\circ F^eV}"]
  & 0 \\
0 \arrow[r]
  & F^{e+n-1}_*R \arrow[r,"{\cdot\,F^e_*\!\left[c\right]\circ V^{n-1}}"]
  & F^e_*W_n(R) \arrow[r,"Res"]
  & F^e_*W_{n-1}(R) \arrow[r]
  & 0,
\end{tikzcd}
\]
where we regard $W_{n-2}(R)=0$ and $V^{n-2}=\mathrm{id}$ if $n=2$,
the snake lemma yields an exact sequence
\[
0 \to F^{n-1}_*B^e_{R}
  \xrightarrow{\ \cdot\,F^e_*\!\left[c\right]\circ V^{n-1}\ }
  Q^e_{R,c,n}
  \xrightarrow{Res}
  Q^e_{R,c,n-1}
  \to 0,
\]
where $B^e_R$ denotes the cokernel of the homomorphism $F^e \colon R \to F^e_*R$.
By the induction hypothesis, it is enough to show that $B^e_R$ is Cohen-Macaulay.
Since $R$ is $F$-pure and we have the exact sequence
\[
0 \to R \to F^e_*R \to B^e_R \to 0,
\]
the $R$-module $B^e_R$ is Cohen-Macaulay.
\end{proof}

\begin{proposition}\label{prop:qF^esplit-adj-inv-adj}
Let $(R,\m)$ be a regular $F$-finite local $\F_p$-algebra and $f \in R^\circ$.
For integers $n,e \geq 1$, the ring $R/f$ is $n$-quasi-$F^e$-split if and only if there exists a $W_n(R)$-module homomorphism $\varphi \colon F^e_*W_n(R) \to W_n\omega_R$ such that
\begin{enumerate}
    \item $\varphi$ induces a homomorphism $Q^e_{R,n} \to W_n\omega_R$,
    \item $\varphi(F^e_*W_n(fR)) \subseteq [f]\cdot W_n\omega_R$, and
    \item $\varphi(F^e_*1) \notin T^{n-1}(\m \cdot \omega_R)$,
\end{enumerate}
where $T^{n-1}_R \colon \omega_R \to W_n\omega_R$ is the Grothendieck trace map of $Res \colon W_n(R) \to R$.
\end{proposition}

\begin{proof}
Consider the following commutative diagram whose rows are exact:
\[
\begin{tikzcd}
0 \arrow[r] & Q^{f,e}_{R,n} \arrow[r] & Q^e_{R,n} \arrow[r] & Q^e_{R/f,n} \arrow[r] & 0 \\
0 \arrow[r] & fR \arrow[r] \arrow[u,"\Phi^{f,e}_{R,n}"] & R \arrow[r] \arrow[u,"\Phi^e_{R,n}"] & R/f \arrow[r] \arrow[u,"\Phi^e_{R/f,n}"] & 0.
\end{tikzcd}
\]

Applying $\Hom_{W_n(R)}(-,W_n\omega_R)$, we obtain the following commutative diagram:
\begin{equation}\label{eq:inv}
\begin{tikzcd}[column sep=0.5cm]
    \Hom_{W_n(R)}(Q^{f,e}_{R,n},W_n\omega_R) \arrow[r,twoheadrightarrow] \arrow[d,"(\Phi^{f,e}_{R,n})^*"] &
    \Ext^1_{W_n(R)}(Q^{e}_{R/f,n},W_n\omega_R) \arrow[d] \arrow[r,"\sim"] &
    \Hom_{W_n(R/f)}(Q^{e}_{R/f,n},W_n\omega_{R/f}) \arrow[d,"(\Phi^e_{R/f,n})^*"]  \\
    \Hom_{W_n(R)}(fR,W_n\omega_R) \arrow[r] &
    \Ext^1_{W_n(R)}(R/f,W_n\omega_R) \arrow[r,"\sim"] &
    \Hom_{W_n(R/f)}(R/f,W_n\omega_{R/f}).
\end{tikzcd}
\end{equation}

The surjectivity of the upper left horizontal map follows from \cref{CM} and the two right horizontal isomorphisms are given by 
\cite{stacks-project}*{\href{https://stacks.math.columbia.edu/tag/0A7I}{Lemma 0A7I}}.

We note that
\[
\Hom_{W_n(R/f)}(R/f,W_n\omega_{R/f}) \simeq \omega_{R/f}
\]
by Grothendieck duality. 
Hence, $R/f$ is $n$-quasi-$F^e$-split if and only if there exists a homomorphism 
$\varphi \colon Q^e_{R/f,n} \to W_n\omega_{R/f}$ such that
\[
(\Phi^e_{R/f,n})^*(\varphi) \notin \m \cdot \omega_{R/f}.
\]
Let $\alpha$ denote the composition of the lower horizontal maps in \eqref{eq:inv} together with the isomorphism above:
\[
\alpha \colon \Hom_{W_n(R)}(fR,W_n\omega_R)
\longrightarrow
\Hom_{W_n(R/f)}(R/f,W_n\omega_{R/f})
\simeq \omega_{R/f}.
\]

By the surjectivity of the upper left horizontal map in \eqref{eq:inv}, 
the $n$-quasi-$F^e$-splitting of $R/f$ is equivalent to the existence of 
\[
\wt{\varphi} \in \Hom_{W_n(R)}(Q^{f,e}_{R,n},W_n\omega_R)
\]
such that
\[
\alpha \circ (\Phi^{f,e}_{R,n})^*(\wt{\varphi})
\notin \m \cdot \omega_{R/f}.
\]

For
\[
\wt{\varphi}
\in
\Hom_{W_n(R)}(Q^{f,e}_{R,n},W_n\omega_R),
\]
Grothendieck duality for
$Res \colon W_n(R)\to R$ gives a natural identification
\[
\theta \colon \Hom_{W_n(R)}(fR,W_n\omega_R)
\simeq
\Hom_R(fR,\omega_R) \simeq f^{-1}\omega_R,
\]
where the map from right to left is induced by $T_R^{n-1}$.
Let 
\[
\xi:=f \cdot \theta( (\Phi^{f,e}_{R,n})^*(\wt{\varphi})) \in\omega_R.
\]
Then we have
\[
\wt{\varphi}(F^e_*[f^{p^e}])
=
T_R^{n-1}(\xi).
\]

By the definition of $\alpha$, the element
\[
\alpha\bigl((\Phi^{f,e}_{R,n})^*(\wt{\varphi})\bigr)
\in \omega_{R/f}
\]
is the image of $\xi$ under the natural surjection
\[
\omega_R
\longrightarrow
\omega_R/f\omega_R
\simeq
\omega_{R/f}.
\]
Indeed, after identifying $\Hom_R(fR,\omega_R)$ with $\omega_R$
by evaluation at $f$, the exact sequence obtained from
$0\to fR\to R\to R/f\to0$ is
\[
0
\longrightarrow
\omega_R
\xrightarrow{\,\cdot f\,}
\omega_R
\longrightarrow
\omega_{R/f}
\longrightarrow
0.
\]
Since $f\in\m$ and $T_R^{n-1}$ is injective, we have
\[
\alpha\bigl((\Phi^{f,e}_{R,n})^*(\wt{\varphi})\bigr)
\notin \m\omega_{R/f}
\quad\Longleftrightarrow\quad
\wt{\varphi}(F^e_*[f^{p^e}])
\notin T_R^{n-1}(\m\omega_R).
\]
Consequently, $R/f$ is $n$-quasi-$F^e$-split if and only if
there exists
\[
\wt{\varphi}\colon
Q^{f,e}_{R,n}
\longrightarrow
W_n\omega_R
\]
such that
\[
\wt{\varphi}(F^e_*[f^{p^e}])
\notin T_R^{n-1}(\m\omega_R).
\]

Precomposing with the natural surjection
\[
F^e_*W_n(fR)\longrightarrow Q^{f,e}_{R,n},
\]
we regard $\wt{\varphi}$ as a homomorphism on
$F^e_*W_n(fR)$.
A homomorphism
\[
\varphi\colon
F^e_*W_n(R)
\longrightarrow
W_n\omega_R
\]
satisfying condition~(2) determines a unique homomorphism
$\wt{\varphi}$ such that
\[
\varphi(y)
=
[f]\cdot\wt{\varphi}(y)
\qquad
\text{for every }y\in F^e_*W_n(fR).
\]
Here we use the injectivity of multiplication by $[f]$ on
$W_n\omega_R$.

Conversely, from a homomorphism
\[
\wt{\varphi}\colon
F^e_*W_n(fR)
\longrightarrow
W_n\omega_R,
\]
we define
\[
\varphi(x)
:=
\wt{\varphi}([f]\cdot x)
\qquad
\text{for }x\in F^e_*W_n(R).
\]
This is well-defined because
\[
[f]\cdot F^e_*a
=
F^e_*([f^{p^e}]a)
\in
F^e_*W_n(fR).
\]
These two constructions are inverse to each other.
Moreover, by the defining pushout diagrams, $\varphi$ factors through
$Q^e_{R,n}$ if and only if $\wt{\varphi}$ factors through
$Q^{f,e}_{R,n}$.
Finally,
\[
\varphi(F^e_*1)
=
\wt{\varphi}([f]\cdot F^e_*1)
=
\wt{\varphi}(F^e_*[f^{p^e}]).
\]
Thus the preceding criterion is equivalent to the existence of a
homomorphism $\varphi$ satisfying conditions~(1)--(3), as desired.
\qedhere
\end{proof}

\begin{lemma}\label{qFr-test-elem}
Let $(R,\m)$ be a normal Gorenstein $F$-finite local $\F_p$-algebra.
Let $t \in \tau(R)$ and $c \in (t^4) \cap R^\circ$, where $\tau(R)$ is the test ideal of $R$.
If there exists an integer $e \geq 1$ such that the homomorphism
\[
(\Phi^{e}_{R,c,n})^* \colon \Hom_{W_n(R)}(Q^{e}_{R,c,n},W_n\omega_R) \to \Hom_{W_n(R)}(R,W_n\omega_R)
\]
is surjective, then $R$ is $n$-quasi-$F$-regular.
\end{lemma}

\begin{proof}
The assertion follows from \cite{KTTWYY3}*{Theorem~4.31}.
\qedhere

\end{proof}

\begin{proposition}\label{qFr-inv-adj}
Let $(R,\m)$ be a regular $F$-finite local $\F_p$-algebra and $f \in R^\circ$.
For an integer $n$ with $n \geq 1$, $t \in \tau(R/f)$ and $c \in (t^4) \cap (R/f)^\circ$, $R/f$ is $n$-quasi-$F$-regular if and only if there exist an integer $e \geq 1$ and a $W_n(R)$-module homomorphism $\varphi \colon F^e_*W_n(R) \to W_n\omega_R$ such that
\begin{enumerate}
    \item $\varphi$ induces a homomorphism $Q^{e}_{R,c,n} \to W_n\omega_R$,
    \item $\varphi(F^e_*W_n(fR)) \subseteq ([f])W_n\omega_R$, and
    \item $\varphi(F^e_*[c]) \notin T^{n-1}_R(\m \cdot \omega_R)$.
\end{enumerate}
\end{proposition}

\begin{proof}
The assertion follows from an argument similar to an argument for the proof of \cref{prop:qF^esplit-adj-inv-adj} and \cref{qFr-test-elem}.
\end{proof}

\section{Trace map of the ghost component}
\begin{notation}\label{notation:F-lift}
Let $(A,\n)$ be a regular local ring such that $A$ is $p$-torsion free.
Let $\phi \colon A \to A$ be a lift of Frobenius such that $\phi$ is finite.
Following \cite{Serre}*{II,\S{6}}, we define the ring homomorphism
\[
\varphi:=\prod_{r \geq 0} \varphi_r \colon W(A) \to \prod_{r \geq 0} A
\]
by
\[
\varphi_{n-1}((a_i)_i)=a_0^{p^{n-1}}+pa_1^{p^{n-2}}+\cdots+p^{n-1}a_{n-1}=\sum_{s=0}^{n-1}p^sa_s^{p^{n-1-s}}.
\]
Furthermore, the ring homomorphism $\varphi$ induces a ring homomorphism
\[
W_n(A) \to \prod_{0 \le r \le n-1} A,
\]
which is also denoted by $\varphi$.
\end{notation}

\begin{lemma}\label{lem:topology on Witt ring}
We use \cref{notation:F-lift}.
Then we have $\varphi(W_n(p^mA))=\bigoplus_{i=0}^{n-1} p^{m+i}A$ for all integers $n,m \geq 1$.
\end{lemma}

\begin{proof}
We take integers $n,m \geq 1$.
The image of $(p^ma_0,\cdots,p^ma_{n-1})$ by $\varphi$ is
\[
(p^ma_0,(p^ma_0)^p+p^{m+1}a_1,(p^ma_0)^{p^2}+p(p^ma_1)^p+p^{m+2}a_2,\cdots ),
\]
thus $\varphi(W_n(p^mA))$ is contained in $\bigoplus_{i=0}^{n-1} p^{m+i}A$.
On the other hand, we assume that the image of $\alpha:=(a_0,a_1,\ldots)$ by $\varphi$ is contained in $\bigoplus_{i=0}^{n-1} p^{m+i}A$.
We prove $a_i \in p^mA$ by the induction on $i$.
For $i=0$, since the first component of $\varphi(\alpha)$ is $a_0$, we have $a_0 \in p^mA$.
Suppose $i \geq 1$.
The $i$-th component of $\varphi(\alpha)$  
\[
a_0^{p^i}+pa_1^{p^{i-1}}+\cdots+p^{i}a_i,
\]
is contained in $p^{m+i}A$.
By the induction hypothesis, we have $a_i \in p^mA$.
\end{proof}

\begin{lemma}\label{lem: V and w}
We use \cref{notation:F-lift}.
Then we have
\[
\varphi_{r}(V\alpha)=p\varphi_{r-1}(\alpha)
\]
for every integer $r \geq 1$ and $\alpha \in W(A)$.
\end{lemma}

\begin{proof}
Let $\alpha=(a_0,a_1,\ldots)$.
Then we have
\begin{align*}
    \varphi_{r}(V\alpha)
    =pa_0^{p^{r-1}}+\cdots +p^{r}a_{r-1} =p\varphi_{r-1}(\alpha),
\end{align*}
as desired.
\end{proof}

\begin{notation}\label{notation:dualizing}
We use \cref{notation:F-lift}.
Let $\omega_A$ be a dualizing module on $A$.
Since $\phi$ is a lift of Frobenius, we have the ring homomorphism $s_\phi \colon A \to W_n(A)$ as in \cite{Yoshikawa25}*{Proposition~3.2}.
We regard $W_n(A)$ as an $A$-module via $s_\phi$.
We define dualizing modules on $W_n(A)$ by 
\[
W_n\omega_A :=\Hom_{A}(W_n(A),\omega_A),
\]
and on $A/p^n$ by
\[
\omega_{A/p^n}:=\mathrm{Ext}^1_{A}(A/p^n,\omega_A).
\]
Furthermore, we fix an isomorphism $\iota \colon A \xrightarrow{\sim} \omega_A$.
The isomorphism induces the isomorphisms $\iota_n \colon A/p^n \xrightarrow{\sim} \omega_{A/p^n}$ for every integer $n \geq 1$.
Furthermore, the Grothendieck trace map $W_n\omega_A \to W_{n+1}\omega_A$ is denoted by $T$ for every $n \geq 1$.
Moreover, the $A$-module homomorphism
\[
\phi_*A/p^n \xrightarrow{\phi_*\iota_n} \phi_*\omega_{A/p^n} \xrightarrow{\phi^*} \omega_{A/p^n} \xrightarrow{\iota_n^{-1}} A/p^n
\]
is denoted by $u$ for every integer $n \geq 1$.
By abuse of notation, we also denote by $u$ the $A$-module homomorphism
\[
\phi_*A \xrightarrow{\iota} \phi_*\omega_A \xrightarrow{\phi^*} \omega_A \xrightarrow{\iota^{-1}} A.
\]
\end{notation}

\begin{proposition}\label{prop:well-def w_n}
We use \cref{notation:dualizing}.
The map $\varphi_{n-1} \colon W_{n}(A) \to A$ induces a ring homomorphism
\[
\varphi_{n-1} \colon W_n(A/p) \to A/p^{n}.
\]
Furthermore, the map is an $A$-module homomorphism if we regard $\varphi_{n-1}$ as 
\[
\varphi_{n-1} \colon W_n(A/p) \to \phi^{n-1}_*A/p^{n}.
\]
\end{proposition}

\begin{proof}
Since we have
\[
\varphi_{n-1}(pa_0,pa_1,\ldots,pa_{n-1}) \in (p^{n})
\]
by \cref{lem:topology on Witt ring}, the homomorphism $\varphi_{n-1}$ induces the map
\[
\varphi_{n-1} \colon W_n(A/p) \to A/p^{n}. 
\]
The second assertion follows from the fact that the homomorphism
\[
\varphi_{n-1} \colon W_n(A) \to \phi^{n-1}_*A
\]
is an $A$-module homomorphism.
\end{proof}

\begin{proposition}\label{prop:w dual}
We use \cref{notation:dualizing}.
The Grothendieck trace map induced by $\varphi_{n-1}$ is denoted by
\[
\varphi_{n-1}^* \colon \phi^{n-1}_*\omega_{A/p^n} \to W_n\omega_{A/p}.
\]
Then we have the following properties:
\begin{enumerate}
    \item $V^r[b]\cdot \varphi_{n-1}^*(\phi_*^{n-1}\omega)=p^r\varphi_{n-1}^*\left(\phi_*^{n-1}(b^{p^{n-1-r}}\cdot \omega)\right)$ for every $\omega \in \omega_{A/p^n}$, $b \in A$ and integer $r \geq 0$.
    \item We have the following commutative diagram;
    \[
    \begin{tikzcd}
    \phi^{n-1}_*\omega_{A/p^{n-1}} \arrow[r, "\phi^*"] \arrow[d, "\pi^*"] &
    \phi^{n-2}_*\omega_{A/p^{n-1}} \arrow[r, "\varphi_{n-2}^*"] &
    W_{n-1}\omega_{A/p} \arrow[d, "T"] \\
    \phi^{n-1}_*\omega_{A/p^n} \arrow[rr, "\varphi_{n-1}^*"] &&
    W_n\omega_{A/p}
    \end{tikzcd}
    \]
    for any integer $n \geq 2$,
    where $\pi^*$ is the dual of the natural surjection $\pi \colon A/p^n \to A/p^{n-1}$.
    \item We have the commutative diagram in which each horizontal sequence is exact
    \[
    \begin{tikzcd}
    0 \arrow[r] & \phi^{n-1}_*\omega_{A/p} \arrow[r, "\pi^*"] \arrow[d, "(\phi^{n-1})^*"] & \phi^{n-1}_*\omega_{A/p^{n}} \arrow[r, "\cdot p"] \arrow[d, "\varphi_{n-1}^*"] & \phi^{n-1}_*\omega_{A/p^{n-1}} \arrow[d, "\varphi_{n-2}^*"] \arrow[r] & 0 \\
    0 \arrow[r] & \omega_{A/p} \arrow[r, "{T^{n-1}}"] & W_{n}\omega_{A/p} \arrow[r, "V^*"] & \phi_*W_{n-1}\omega_{A/p} \arrow[r] & 0
    \end{tikzcd}
    \]
\end{enumerate}
\end{proposition}

\begin{proof}
Since $\varphi_{n-1} \colon W_n(A/p) \to {\phi^{n-1}}_*A/p^n$ is a $W_n(A)$-module homomorphism, the $W_n\omega_A$-dual of it is also a $W_n(A)$-module homomorphism.
Therefore, we have
\begin{align*}
V^r[b]\cdot \varphi_{n-1}^*(\phi_*^{n-1}\omega)
&=\varphi_{n-1}^*\left(\phi_*^{n-1}(\varphi_{n-1}(V^r[b])\cdot \omega)\right) 
=\varphi_{n-1}^*\left(\phi_*^{n-1}(p^{r}b^{p^{n-1-r}}\cdot \omega)\right) \\
&=p^r\varphi_{n-1}^*\left(\phi_*^{n-1}(b^{p^{n-1-r}}\cdot \omega)\right),
\end{align*}
thus we obtain Assertion (1).
Assertion (2) follows from the following commutative diagram
\[
\begin{tikzcd}
W_n(A/p) \arrow[rr, "\varphi_{n-1}"] \arrow[d, "Res"] & &
\phi^{n-1}_*A/p^n \arrow[d,"\pi"] \\
W_{n-1}(A/p) \arrow[r, "\varphi_{n-2}"] &
\phi^{n-2}_*A/p^{n-1} \arrow[r, "\phi"] &
\phi^{n-1}_*A/p^{n-1}.
\end{tikzcd}
\]
Indeed, for $\alpha=(a_0,\ldots, a_{n-1}) \in W_{n}(A/p)$, if we take lifts $\widetilde{a}_i \in A$ of $a_i$, then we have
\begin{align*}
    \phi \circ \varphi_{n-2} \circ Res(\alpha)
    &=\phi\left(\phi_*^{}(\widetilde{a}_0^{p^{n-2}}+\cdots +p^{n-2}\widetilde{a}_{n-2})\right) \\
    &\equiv \widetilde{a}_0^{p^{n-1}}+\cdots +p^{n-2}\widetilde{a}_{n-2}^p \mod p^{n-1} \\
    &\equiv \varphi_{n-1}(\alpha) \mod p^{n-1}.
\end{align*}

Assertion (3) follows from the following commutative diagram in which each horizontal sequence is exact:
\[
\begin{tikzcd}
0 \arrow[r] &
\phi_*W_{n-1}(A/p) \arrow[r, "V"] \arrow[d, "\varphi_{n-2}"] &
W_n(A/p) \arrow[r, "(Res)^{n-1}"] \arrow[d, "\varphi_{n-1}"] &
A/p \arrow[r] \arrow[d, "\phi^{n-1}"] &
0 \\
0 \arrow[r] &
\phi^{n-1}_*A/p^{n-1} \arrow[r, "\cdot p"] &
\phi^{n-1}_*A/p^n \arrow[r] &
\phi^{n-1}_*A/p \arrow[r] &
0.
\end{tikzcd}
\]
Indeed, the commutativity follows from Lemma \ref{lem: V and w} and
\[
\phi^{n-1} \circ (Res)^{n-1}(\alpha)\equiv a_0^{p^{n-1}}=\varphi_{n-1}(\phi_*^{n-1}\alpha) \mod p
\]
for $\alpha=(a_0,\ldots,a_{n-1}) \in W_n(A/p)$.
\end{proof}

\begin{proposition}\label{thm:F-inverse action'}
We use \cref{notation:dualizing}.
We define the $A$-module homomorphism
\[
\lambda_n \colon \phi^{n-1}_*A/p^n \xrightarrow{\phi^{n-1}_*\iota_n} \phi^{n-1}_*\omega_{A/p^n} \xrightarrow{\varphi_{n-1}^*} W_n\omega_{A/p}
\]
for every integer $n \geq 1$.
Then $\lambda_n$ satisfies the following properties;
\begin{itemize}
    \item[(C1)] We have $V^r[b] \cdot \lambda_n\left(\phi_*^{n-1}a\right)=p^r \lambda_n\left(\phi_*^{n-1}(ab^{p^{n-1-r}})\right)$ for $a,b \in A/p^n$, $r \geq 0$, and $n \geq 1$.
    \item[(C2)] We have $\lambda_n\left(\phi_*^{n-1}(pa)\right)=T \circ \lambda_{n-1}\circ u\left(\phi_*^{n-1}a\right)$ for $a \in A/p^{n-1}$ for every $n \geq 2$.
    \item[(C3)] We have the following commutative diagram in which horizontal sequence is exact;
    \[
    \begin{tikzcd}
    0 \arrow[r] &
    \phi^{n-1}_*A/p \arrow[r, "\cdot p^{n-1}"] \arrow[d, "\lambda_1 \circ u^{n-1}"] &
    \phi^{n-1}_*A/p^n \arrow[r,"\phi^{n-1}_*\pi"] \arrow[d, "\lambda_{n}"] &
    \phi^{n-1}_*A/p^{n-1} \arrow[r] \arrow[d, "\lambda_{n-1}"] &
    0 \\
    0 \arrow[r] &
    \omega_{A/p} \arrow[r, "{T^{n-1}}"] &
    W_n\omega_{A/p} \arrow[r, "V^*"] &
    \phi_*W_{n-1}\omega_{A/p} \arrow[r] &
    0.
    \end{tikzcd}
    \]
    for every $n \geq 2$.
\end{itemize}
\end{proposition}

\begin{proof}
Assertion \cond{C1} follows from \cref{prop:w dual} (1).
Next, we prove Assertion \cond{C2}.
We have the following commutative diagram
\begin{equation}\label{eq:p and a}
    \begin{tikzcd}
    A/p^{n-1} \arrow[r, "\cdot p"] \arrow[d,"\iota_{n-1}"] & A/p^n \arrow[d,"\iota_n"] \\
    \omega_{A/p^{n-1}} \arrow[r, "\pi^*"] & \omega_{A/p^n}.
    \end{tikzcd}
\end{equation}
Indeed, we consider the following commutative diagram in which each row is exact;
\[
\begin{tikzcd}
0 \arrow[r] & A \arrow[r, "\cdot p^n"] \arrow[d, "\cdot p"] & A \arrow[r] \arrow[d, equal] & A/p^n \arrow[r] \arrow[d] & 0 \\
0 \arrow[r] & A \arrow[r, "\cdot p^{n-1}"] & A \arrow[r] & A/p^{n-1} \arrow[r] & 0.
\end{tikzcd}
\]
Taking $\omega_A$-dual, we have
\[
\begin{tikzcd}
0 \arrow[r] & \omega_A \arrow[r, "\cdot p^{n-1}"] \arrow[d, equal] & \omega_A \arrow[r] \arrow[d, "\cdot p"] & \omega_{A/p^{n-1}} \arrow[r] \arrow[d, "\pi^*"] & 0 \\
0 \arrow[r] & \omega_A \arrow[r, "\cdot p^n"] & \omega_A \arrow[r] & \omega_{A/p^n} \arrow[r] & 0,
\end{tikzcd}
\]
thus, we obtain the commutative diagram (\ref{eq:p and a}).
Therefore, we obtain
\begin{align*}
    \lambda_n\left(\phi_*^{n-1}(pa)\right)
    &=\varphi_{n-1}^* \circ \iota_n\left(\phi_*^{n-1}(pa)\right)
    =\varphi_{n-1}^* \circ  \pi^* \circ \iota_{n-1}\left(\phi_*^{n-1}a\right) \\
    &= T \circ \varphi_{n-2}^* \circ \phi^* \circ \iota_{n-1}\left(\phi_*^{n-1}a\right) \\
    &=T \circ \varphi_{n-2}^* \circ \iota_{n-1} \circ  u\left(\phi_*^{n-1}a\right) \\
    &=T \circ \lambda_{n-1} \circ u\left(\phi_*^{n-1}a\right),
\end{align*}
thus we obtain Assertion \cond{C2}.
Finally, Assertion \cond{C3} follows from Assertions \cond{C2} and (3) in \cref{prop:w dual}.
\end{proof}

\section{Structure of $\Hom_{W_n(A/p)}(F^e_*W_n(A/p),W_n\omega_{A/p})$}

\begin{proposition}\label{prop:description of trace}
We use \cref{notation:dualizing}.
We define an $A$-module homomorphism $\Psi^e_n$ by the composition
\begin{align*}
    \phi^{e+n-1}_*A/p^n \xrightarrow{\phi^e_*\lambda_n} \phi^e_*W_n\omega_{A/p} \xrightarrow{\sim} \Hom_{W_n(A/p)}(F^e_*W_n(A/p),W_n\omega_{A/p}),
\end{align*}
where the second arrow is the isomorphism given by Grothendieck duality. 
More precisely, we have natural isomorphisms
\begin{align*}
    \Hom_{W_n(A/p)}(F^e_*W_n(A/p),W_n\omega_{A/p}) 
    &\simeq 
    F^e_*\Hom_{W_n(A/p)}(W_n(A/p),W_n\omega_{A/p}) \\
    &\simeq 
    F^e_*W_n\omega_{A/p} \\
    &= 
    \phi^e_*W_n\omega_{A/p}.
\end{align*}
Then, for $g \in A/p^n$, the homomorphism $\Psi^e_n\paren{\phi_*^{e+n-1}g}$ coincides with the composition
\[
F^e_*W_n(A/p) \xrightarrow{\varphi_{n-1}} \phi^{e+n-1}_*A/p^n \xrightarrow{u^e\left(\phi^e_*(g\cdot-)\right)} \phi^{n-1}_*A/p^n \xrightarrow{\lambda_n} W_n\omega_{A/p}.
\]
\end{proposition}

\begin{proof}
Let $g \in A/p^n$.
Then, by the isomorphism 
\[
\Hom_{W_n(A/p)}(F^e_*W_n(A/p),W_n\omega_{A/p}) \simeq F^e_*\Hom_{W_n(A/p)}(W_n(A/p),W_n\omega_{A/p}),
\]
the homomorphism $\psi^e_{n,g}:=\Psi^e_n\paren{\phi_*^{e+n-1}g}$ corresponds to the map $\psi$ such that the diagram
\[
\begin{tikzcd}
F^e_*W_n(A/p) \arrow[r, "\psi^e_{n,g}"] \arrow[d, "\psi"] & W_n\omega_{A/p} \\
F^e_*W_n\omega_{A/p} \arrow[ru,"T_{F^e}"']
\end{tikzcd}
\]
commutes, where $T_{F^e}$ is the Grothendieck trace map with respect to $F^e \colon W_n(A/p) \to F^e_*W_n(A/p)$.
Furthermore, by the isomorphism
\[
\Hom_{W_n(A/p)}(W_n(A/p),W_n\omega_{A/p}) \simeq W_n\omega_{A/p},
\]
$\psi$ corresponds to $\psi(1) \in W_n\omega_{A/p}$.
Then $\psi$ satisfies $\psi(\alpha)=\alpha\cdot \psi(1)$ for $\alpha \in W_n(A/p)$.
Now, since $\psi(1)=\lambda_n\left(\phi_*^{n-1}g\right)$ by the definition of $\Psi^e_n$, we have
\[
\psi(\alpha)=\alpha \cdot \psi(1)=\alpha \cdot \lambda_n(\phi_*^{n-1}g)=\lambda_n\left(\phi_*^{n-1}(g\varphi_{n-1}(\alpha))\right).
\]
We consider the following commutative diagram
\[
\begin{tikzcd}
F^e_*W_n\omega_{A/p} \arrow[r,"T_{F^e}"] & W_n\omega_{A/p} \\
\phi^{e+n-1}_*(A/p^n) \arrow[r, "u^e"] \arrow[u, "\phi^e_*\lambda_n"] & \phi^{n-1}_*A/p^n, \arrow[u, "\lambda_n"]
\end{tikzcd}
\]
then the image of $\psi(\alpha)=\lambda_n\left(\phi_*^{n-1}(g\varphi_{n-1}(\alpha))\right)$ by $T_{F^e}$ is
\[
\lambda_n \circ u^e\left(\phi_*^{e+n-1}(g\varphi_{n-1}(\alpha))\right),
\]
as desired.
\end{proof}

\begin{proposition}\label{prop:psi^e_n is module hom}
We use \cref{notation:dualizing}.
Let $g \in A/p^n$ and $\psi^e_{n,g}:=\Psi^e_n\paren{\phi_*^{e+n-1}g}$ for all $e,n$.
Then we have the following formulas:
\begin{enumerate}
    \item $\psi^e_{n,g}\paren{F_*^{e}(V^r[a])}=T^r \circ \lambda_{n-r} \circ u^{e+r}\left(\phi_*^{e+n-1}(ga^{p^{n-r-1}})\right)$ for $n,e \geq 1$, $n-1 \geq r \geq 0$,  $a \in A/p^n$,
    \item $T^s \circ \psi^{e+s}_{n-s,g}\left(F_*^{e+s}(F^s\alpha F^e\beta)\right)
    =(V^s\beta)\psi^e_{n,g}\left(F_*^{e}\alpha\right)$
    for $n,e \geq 1$, $n-1 \geq s \geq 0$, $\alpha,\beta \in W_n(A/p)$, 
    \item $\psi^e_{n,g}\left(F^e_*(V\alpha)\right)=T \circ \psi^{e+1}_{n-1,g}\left(F^{e+1}_*\alpha\right)$ for $n \geq 2$, $e \geq 1$, $\alpha \in W_n(A/p)$, and
    \item $\psi^e_{n,g}\left(F^e_*(F\alpha)\right)=\psi^{e-1}_{n,u\left(\phi_*g\right)}\left(F^{e-1}_*\alpha\right)$ for $e \geq 2$, $n \geq 1$, $\alpha \in W_n(A/p)$.
\end{enumerate}
\end{proposition}

\begin{proof}
First, we have
\begin{align*}
    \psi^e_{n,g}\left(F^e_*(V^r[a])\right) 
    &\overset{(\star_1)}{=} \lambda_n \circ u^e\left(\phi^{e+n-1}_*(gp^ra^{p^{n-1-r}})\right) \\
    &\overset{(\star_2)}{=} T^r \circ \lambda_{n-r} \circ u^{e+r}\left(\phi_*^{e+n-1}(ga^{p^{n-1-r}})\right),
\end{align*}
where $(\star_1)$ follows from \cref{prop:description of trace} and $(\star_2)$ follows from \cond{C2} in \cref{thm:F-inverse action'}, thus we obtain Assertion (1).
Next, we have
\begin{align*}
T^s \circ \psi^{e+s}_{n-s,g}\left(F^{e+s}_*(F^s\alpha F^e\beta)\right) 
&\eqtag{3} T^s \circ \lambda_{n-s} \circ u^{e+s}\left(\phi_*^{e+n-1}(g\varphi_{n-s-1}(F^s\alpha F^e\beta))\right) \\
&\smash{\overset{(\star_4)}{=}} \lambda_n \circ u^{e}\left(\phi^{e+n-1}_*(p^s g \varphi_{n-s-1}(F^s\alpha F^e\beta))\right) \\
&\smash{\overset{(\star_5)}{=}} \lambda_n \circ u^{e}\left(\phi^{e+n-1}_*(g \varphi_{n-1}(V^s(F^s \alpha F^e\beta))) \right) \\
&= \lambda_n \circ u^{e}\left(\phi^{e+n-1}_*(g \varphi_{n-1}(\alpha V^sF^{e}\beta))\right) \\
&= \psi^e_{n,g}\left(F^{e}_*(\alpha V^sF^e\beta)\right) \\
&= (V^s \beta) \psi^e_{n,g}\left(F_*^{e}\alpha\right)
\end{align*}
where $(\star_3)$ follows from \cref{prop:description of trace}, $(\star_4)$ follows from \cond{C2} in \cref{thm:F-inverse action'}, and $(\star_5)$ follows from \cref{lem: V and w}, thus we obtain Assertion (2).
Assertion (3) follows from the computation
\begin{align*}
    \psi^e_{n,g}\left(F^e_*(V\alpha)\right)
    &=  \lambda_n\circ u^e\left(\phi^{e+n-1}_*(g\varphi_{n-1}(V\alpha))\right) \\
    &\eqtag{6}  \lambda_n \circ u^e\left(\phi^{e+n-1}_*(pg\varphi_{n-2}(\alpha))\right) \\
    &\eqtag{7} T \circ \lambda_{n-1} \circ u^{e+1}\left(\phi^{e+n-1}_*(g\varphi_{n-2}(\alpha))\right) \\
    &=T \circ \psi^{e+1}_{n-1,g}\left(\phi^{e+n-1}_*\alpha\right),
\end{align*}
where $(\star_6)$ follows from \cref{lem: V and w} and $(\star_7)$ follows from \cond{C2} in \cref{thm:F-inverse action'}.
Finally, we prove Assertion (4).
For $\alpha=(a_0,\ldots,a_{n-1}) \in W_n(A/p)$, we have
\begin{align*}
    \psi^e_{n,g}\left(F^e_*(F\alpha)\right) 
    &= \lambda_n \circ u^e\left(\phi_*^{e+n-1}(g\varphi_{n-1}(F\alpha))\right) \\
    &= \lambda_n\left(u^e\left(\phi_*^{e+n-1}(g(a_0^{p^n}+pa_{1}^{p^{n-1}}+\cdots+p^{n-1}a_{n-1}^p))\right)\right) \\
    &= \lambda_n\left(u^{e-1}\left(\phi_*^{e+n-2}(u(\phi_*g)(a_0^{p^{n-1}}+\cdots+p^{n-1}a_{n-1}))\right)\right) \\
    &= \psi^{e-1}_{n,u\left(\phi_*g\right)}\left(\phi^{e+n-2}_*\alpha\right),
\end{align*}
as desired.
\end{proof}

\begin{proposition}\label{thm:structure of dual'}
We use \cref{notation:dualizing}.
The map 
\[
\Psi^e_n \colon \phi^{e+n-1}_*A/p^n \to (F^e_*W_n(A/p))^*:=\Hom_{W_n(A/p)}(F^e_*W_n(A/p),W_n\omega_{A/p})
\]
is an $A$-module homomorphism and satisfies the following conditions:
\begin{enumerate}
    \item We have $\Psi^e_n\left(\phi_*^{e+n-1}(pg)\right)=T \circ \Psi^e_{n-1}\left(\phi_*^{e+n-2}u(\phi_*g)\right)$ for $n \geq 2$, $e \geq 1$, $g \in A/p^n$. 
    \item We have the commutative diagram in which each horizontal sequence is exact;
    \[
    \begin{tikzcd}
    0 \arrow[r] & F^{e+n-1}_*A/p \arrow[r, "\cdot p^{n-1}"] \arrow[d, "\Psi^e_1 \circ u^{n-1}"'] & \phi^{e+n-1}_*A/p^n \arrow[r] \arrow[d, "\Psi^e_n"'] & \phi^{e+n-1}_*A/p^{n-1} \arrow[r] \arrow[d, "\Psi^{e+1}_{n-1}"'] & 0 \\
    0 \arrow[r] & (F^e_*A/p)^* \arrow[r, "{(F^e_*Res)^*}"'] & (F^e_*W_n(A/p))^* \arrow[r, "{(F^e_*V)^*}"'] & (F^{e+1}_*W_{n-1}(A/p))^* \arrow[r] & 0
    \end{tikzcd}
    \]
    for $n \geq 2$, $e \geq 1$.
    In particular, $\Psi^e_n$ is surjective for all $e,n$.
    \item For $n, e \geq 1$,  $g \in A/p^n$, $g \in \mathrm{ker}(\Psi^e_n)$ if and only if $u^{r-1}\left(\phi^{r-1}_*g\right) \in (p^r)$ for all $n \geq r \geq 1$.
    \item We have $\Psi^e_n\left(\phi^{e+n-1}_*(gf^{p^{e+n-1}})\right)=[f]\Psi^e_n\left(\phi^{e+n-1}_*g\right)$ for $n,e \geq 1$, $f,g \in A/p^n$.
\end{enumerate}
\end{proposition}

\begin{proof}
We note that on $W_n(A/p)$, we have $FV(\alpha)=p\alpha$ for every $\alpha \in W_n(A/p)$ by \cite{Serre}*{II,\S{6}}.
By construction, $\Psi^e_{n}$ is an $A$-module homomorphism.
Therefore, by Proposition \ref{prop:psi^e_n is module hom}, we have
\begin{align*}
    \Psi^e_n\left(\phi^{e+n-1}_*(pg)\right) (F^e_*\alpha)
    &=p\Psi^e_n\left(\phi^{e+n-1}_*g\right)(F^e_*\alpha) \\
    &=T \circ \psi^{e+1}_{n-1,g}\left(F^{e+1}_*(F\alpha)\right) \\
    &= T \circ \psi^e_{n-1,u(\phi_*g)}(\phi^{e+n-2}_*\alpha),
\end{align*}
where the last equality follows from \cref{prop:psi^e_n is module hom}~(4),
thus we obtain Assertion (1).
Assertion (2) then follows from Proposition \ref{prop:psi^e_n is module hom} (2).
The surjectivity follows from the snake lemma and the surjectivity of $u$ and $\Psi^e_1$.
We note that $\Psi^e_1$ is an isomorphism.

Next, we prove Assertion (3) by induction on $n$.
For $n=1$, the assertion follows from the fact that $\Psi^e_1$ is an isomorphism.
We assume $n \geq 2$.
By (2), we have the exact sequence
\[
0 \to \mathrm{ker}(\Psi^e_1 \circ u^{n-1}) \to \mathrm{ker}(\Psi^e_n) \to \mathrm{ker}(\Psi^{e+1}_{n-1}) \to 0.
\]
If  $g \in \mathrm{ker}(\Psi^e_n)$,
then the image in $A/p^{n-1}$ is contained in $\mathrm{ker}(\Psi^{e+1}_{n-1})$.
By the induction hypothesis, we have $g \in (p)$.
Let $g=pg'$.
Then we have 
\[
\Psi^e_n\left(\phi^{e+n-1}_*g\right)=T \circ \Psi^e_{n-1} \circ u\left(\phi^{e+n-1}_*g'\right)
\]
by (1).
Since $T$ is injective, $\Psi^e_n\left(\phi^{e+n-1}_*g\right)=0$ implies $\Psi^e_{n-1} \circ u\left(\phi^{e+n-1}_*g'\right)=0$.
By the induction hypothesis, we have $u^{r}\left(\phi^r_*g'\right) \in (p^r)$ for all $1 \leq r \leq n-2$.
Thus, $u^{r-1}\left(\phi^{r-1}_*g\right) \in (p^{r})$ for all $2 \leq r \leq n-1$.
Combining $g \in (p)$, we obtain $u^{r-1}\left(\phi^{r-1}_*g\right) \in (p^r)$ for all $1 \leq r \leq n-1$, as desired.
Next, we assume $u^{r-1}\left(\phi^{r-1}_*g\right) \in (p^r)$ for $1 \leq r \leq n-1$.
Then, $u^{r}\left(\phi^r_*(g/p)\right) \in (p^r)$ for $1 \leq r \leq n-2$, thus $\Psi^e_n\left(\phi^{e+n-1}_*g\right)=\Psi^e_{n-1} \circ u\left(\phi_*^{e+n-1}(g/p)\right)=0$ by the induction hypothesis, as desired.

Finally, we prove Assertion (4).
We have
\begin{align*}
    &\psi^e_{n,f^{p^{e+n-1}}g}\left(F^e_*(V^r[a])\right)
    \overset{(\star_1)}{=}T^r \circ \lambda_{n-r} \circ u^{e+r}\left(\phi^{e+n-1}_*(gf^{p^{e+n-1}}a^{p^{n-r-1}})\right) \\
    &\overset{(\star_2)}{=}\psi^e_{n,g}\left(F^e_*(V^r[af^{p^{e+r}}])\right)
    =\psi^e_{n,g}\left(F^e_*([f^{p^{e}}]V^r[a])\right)
    =[f]\psi^e_{n,g}\left(F^e_*(V^r[a])\right)
\end{align*}
for every $a \in A/p$, where $(\star_1)$ and $(\star_2)$ follow from \cref{prop:psi^e_n is module hom} (1), thus we obtain Assertion (4).
\end{proof}

\section{Fedder-type criterion for quasi-$F^e$-splitting and quasi-$F$-regularity}

\begin{notation}\label{notation:fedder}
Let $R$ be an $F$-finite regular local $\F_p$-algebra.
Then $R$ has a $p$-basis by \cite{KN1980}*{Corollary~3.2} (see also \cite{KTY}*{Lemma~3.1}).
Furthermore, by \cite{DeJong95}*{Remark~1.2.3}, there exist $p$-torsion free regular local ring $A$ and a lift of Frobenius $\phi \colon A \to A$ such that $A/pA \simeq R$.
Since $A/pA$ is regular, we have a regular system of parameters $p,x_1,\ldots,x_N$.
We set $\m:=(x_1,\ldots,x_N)A$ and $\m^{[p]}:=(x_1^{p},\ldots,x_N^{p})$. 
We note that $(\m,p)$ is the maximal ideal of $A$ and $\m R$ is the maximal ideal of $R$.
We also use \cref{notation:dualizing}.
\end{notation}

\begin{lemma}\label{ideal-contained}
We use \cref{notation:fedder}.
Let $I=(a_1,\ldots,a_r)$ be an ideal of $A$, let $g \in A$, and let $e$ be a positive integer.
Then $g \in (\phi^e(a_1),\ldots,\phi^e(a_r))$ if and only if 
$u^e(F^e_*(gA)) \subseteq I$.
\end{lemma}

\begin{proof}
The ``only if'' direction is clear.
For the converse, assume that $u^e(F^e_*(gA)) \subseteq I$.

Let $v_1,\ldots,v_N$ be lifts of a basis of $F^e_*R$ over $R$.
Since $\phi^e_*A$ is a free $A$-module, the elements $v_1,\ldots,v_N$ also form an $A$-basis of $\phi^e_*A$.
Let $v_1^*,\ldots, v_N^* \in \Hom_A(\phi^e_*A,A)$ be the dual basis corresponding to $v_1,\ldots,v_N$.

We note that
\[
\Hom_A(\phi^e_*A,A)
\simeq
\Hom_A(\phi^e_*A,\omega_A)
\simeq
\phi^e_*\omega_A
\simeq
\phi^e_*A.
\]
Furthermore, the above isomorphism 
$\phi^e_*\omega_A \simeq \Hom_A(\phi^e_*A,\omega_A)$
is obtained as the composition
\[
\phi^e_*\omega_A
\xrightarrow{\sim}
\phi^e_*\Hom_A(A,\omega_A)
\xrightarrow{\sim}
\Hom_A(\phi^e_*A,\omega_A),
\]
where the second isomorphism is given by
\[
\varphi
\longmapsto
\left(
\phi^e_*A
\xrightarrow{\ \phi^e_*\varphi\ }
\phi^e_*\omega_A
\xrightarrow{\ (\phi^e)^*\ }
\omega_A
\right).
\]
Hence $\Hom_A(\phi^e_*A,A)$ is a free $\phi^e_*A$-module of rank one generated by $u^e$.

In particular, for each $i=1,\ldots,N$, there exists $c_i \in A$ such that
\[
v_i^* = u^e(F^e_*(c_i \cdot -)).
\]
Write $g$ in terms of the basis $v_1,\ldots,v_N$ as
\[
g = \sum_{i=1}^{N} v_i \phi^e(g_i).
\]
By assumption, we have
\[
u^e(F^e_*(g c_i))
=
v_i^*(F^e_*g)
=
g_i
\in I
\]
for every $i=1,\ldots,N$.
Therefore,
\[
g
=
\sum_{i=1}^{N} v_i \phi^e(g_i)
\in
\phi^e(I)A
=
(\phi^e(a_1),\ldots,\phi^e(a_r)),
\]
as desired.
\end{proof}

\begin{lemma}\label{lem:including condition}
We use \cref{notation:fedder}.
Let $f \in A$, $a \in A$, and $e \geq 1$.
We assume that $f,p$ is a regular sequence.
Then
\begin{itemize}
    \item[(1)] $T(\omega) \in [f] W_{e+1}\omega_{R}$ if and only if $\omega \in [f] W_e\omega_R$ for $\omega \in W_e\omega_R$, 
    \item[(2)] $\lambda_e \circ u^s\left(\phi^{e+s-1}_*(f^{p^{e+s-1}}a)\right) \in [f] W_e\omega_R$ for all $s \geq 0$, and
    \item[(3)] for $g \in A$, if
    \[
    \lambda_e\left(\phi^{e-1}_*(gA)\right) \subseteq [f]W_e\omega_R,
    \]
    then $g \in (f^{p^{e-1}},p^e)$.
\end{itemize}
\end{lemma}

\begin{proof}
We consider the commutative diagram in which each horizontal sequence is exact;
\[
\begin{tikzcd}
0 \arrow[r] & W_e\omega_R \arrow[r, "T"] \arrow[d, "\cdot \text{[}f\text{]}"'] 
            & W_{e+1}\omega_R \arrow[r,"(V^{e})^*"] \arrow[d, "\cdot \text{[}f\text{]}"'] 
            & F^e_*\omega_R \arrow[r] \arrow[d, "\cdot f"'] 
            & 0 \\
0 \arrow[r] & W_e\omega_R \arrow[r, "T"] 
            & W_{e+1}\omega_R \arrow[r,"(V^{e})^*"] 
            & F^e_*\omega_R \arrow[r] 
            & 0.
\end{tikzcd}
\]
Furthermore, since the map $\cdot f$ is injective, so is multiplication by $[f]$ on $W_e\omega_R$ by inductive argument.
Therefore, we obtain the injectivity of the map
\[
W_e\omega_R/[f]W_e\omega_R \to W_{e+1}\omega_R/[f]W_{e+1}\omega_R,
\]
thus obtaining Assertion (1).
Next, we prove Assertion (2).
We have
\begin{align*}
    & T^s \circ \lambda_e \circ u^s\left(\phi^{e+s-1}_*(f^{p^{e+s-1}}a)\right) \\
    &= p^s \circ \lambda_{e+s}\left(\phi^{e+s-1}_*(f^{p^{e+s-1}}a)\right) &\text{by}\ \cond{C2} \\
    &= [f] \cdot p^s  \lambda_{e+s}\left(\phi^{e+s-1}_*a\right) &\text{by}\ \cond{C1},
\end{align*}
providing Assertion (2).

Next, we prove Assertion (3) by induction on $e$.
Since we have
\begin{align*}
    & p^{e-1} \lambda_e\left(\phi^{e-1}_*(gA)\right) 
    = T^{e-1} \bigl(u^{e-1}\left(\phi^{e-1}_*(gA)\right) \cdot \lambda_1(1) \bigr)  &\text{by}\ \cond{C2},
\end{align*}
which is contained in $[f]W_e\omega_R$ by assumption.
Thus, we obtain $u^{e-1}(\phi^{e-1}_*(gA)) \lambda_1(1) \in f\omega_A$, which implies $u^{e-1}(\phi^{e-1}_*(gA)) \in fA$.
By \cref{ideal-contained}, we obtain $g \in (f^{p^{e-1}},p)$, and in particular, Assertion (3) in the case of $e=1$.
We write $g=f^{p^{e-1}}g'+ph$.
By Assertion (2), we have
\[
p\lambda_e\left(\phi^{e-1}_*(hA)\right) \subseteq [f] \cdot W_e\omega_R.
\]
By the equation
\begin{align*}
    p\lambda_e\left(\phi^{e-1}_*(hA)\right)=T \circ \lambda_{e-1} \circ u\left(\phi^{e-1}_*(hA)\right)
\end{align*}
and Assertion (1), we obtain $u\left(\phi_*(hA)\right) \subseteq (f^{p^{e-2}},p^{e-1})$ by the induction hypothesis, thus $h$ is contained in $(\phi(f^{p^{e-2}}),p^{e-1})$ by \cref{ideal-contained}.
Since we have $f^{p^{e-1}} \equiv \phi(f^{p^{e-2}}) \mod p^{e-1}$, we have $h \in (f^{p^{e-1}},p^{e-1})$, as desired.
\end{proof}

\begin{proposition}\label{prop:pushout condition}
We use \cref{notation:fedder}.
Let $n,e \geq 1$ be integers and $g \in A$.
Then $\psi^e_{n,g}$ is contained in $(Q^e_{R,n})^*$ if and only if $u^{e+r-1}\left(\phi^{e+r-1}_*g\right)$ is contained in $(p^{r})$ for all $n-1 \geq  r \geq 1$.
\end{proposition}

\begin{proof}
Since the kernel of $F_*W_n(R) \to Q_{R,n}$ is mapped isomorphically to that of $F^e_*W_n(R) \to Q^e_{R,n}$ by $F^{e-1}$ thanks to pushout properties, we obtain that $\psi^e_{n,g}$ is contained in $(Q^e_{R,n})^*$ if and only if the homomorphism
\[
\psi \colon F_*W_n(R) \xrightarrow{F^{e-1}} F^e_*W_n(R) \xrightarrow{\psi^e_{n,g}} W_n\omega_R
\]
is contained in $(Q^1_{R,n})^*$, we may assume $e=1$ by \cref{prop:psi^e_n is module hom} (4).
Furthermore, the homomorphism $\psi^1_{n,g}$ is contained in $(Q^1_{R,n})^*$ if and only if $\psi^1_{n,g}\left(F_*(p\alpha)\right)=0$ for every $\alpha \in W_n(R)$.
Moreover, since we have
\[
\psi^1_{n,g}\left(F_*(p\alpha)\right)=T \circ \psi^1_{n-1,u\left(\phi_*g\right)}\left(\phi^{n-1}_*\alpha\right)
\]
by \cref{thm:structure of dual'} and $T$ is injective, $\psi^1_{n,g}$ is contained in $(Q^1_{R,n})^*$ if and only if $u\left(\phi_*g\right) \in \mathrm{ker}(\Psi^1_{n-1})$.
Therefore, the assertion follows from \cref{thm:structure of dual'} (3).
\end{proof}

\begin{proposition}\label{prop:including condition 2}
We use \cref{notation:fedder}.
Let $n,e \geq 1$ and $f,g \in A/p^n$.
We assume $f$ is a non-zero divisor.
Then $\psi^e_{n,g}\left(F^e_*W_n(fR)\right) \subseteq [f] W_n\omega_R$ if and only if $g$ admits a decomposition
\[
g=g_0+pg_1+\cdots +p^{n-1}g_{n-1}
\]
such that $u^r\left(\phi^r_*g_r\right) \in (f^{p^{e+n-r-1}-1})$ for every $n-1 \geq r \geq 0$.
\end{proposition}

\begin{proof}
First, we prove the following assertion.
\begin{claim}\label{cl:inclusion cond}
If $g \in (f^{p^{e+n-1}-1})$, then $\psi^e_{n,g}\left(F^e_*W_n(fR)\right) \subseteq [f] W_n\omega_R$.
\end{claim}
\begin{claimproof}
It is enough to show that
\[
\psi^e_{n,g}\left(F^e_*(V^r[fa])\right) \in [f]W_n\omega_R
\]
for every $a \in R$ and $r \geq 0$.
We take $a \in R$ and $n-1 \geq r \geq 0$, then we have
\[
\psi^e_{n,g}\left(F^e_*(V^r([fa]))\right)=T^r \circ \lambda_{n-r} \circ u^{e+r}\left(\phi^{e+n-1}_*(g(fa)^{p^{n-r-1}})\right)
\]
by \cref{prop:psi^e_n is module hom} (1).
Since $gf^{p^{n-r-1}} \in (f^{p^{e+n-1}})$, there exists $g'$ such that $gf^{p^{n-r-1}}=f^{p^{e+n-1}}g'$.
Furthermore, we have
\begin{align*}
   \lambda_{n-r} \circ u^{e+r}\left(\phi^{e+n-1}_*(g(fa)^{p^{n-r-1}})\right)
   &=\Psi^{e+r}_{n-r}\left(\phi^{e+n-1}_*(f^{p^{e+n-1}}g')\right)\left(F^e_*[a]\right) \\
   &=[f]\Psi^{e+r}_{n-r}\left(\phi^{e+n-1}_*g'\right)\left(F^e_*[a]\right), 
\end{align*}
where the first equation follows from \Cref{prop:description of trace} and the last equation follows from Proposition \ref{thm:structure of dual'} (4).
Therefore, we obtain 
\[
\psi^e_{n,g}\left(F^e_*W_n(fR)\right) \subseteq [f] W_n\omega_R
\]
by \cref{lem:including condition} (1).
\end{claimproof}

\noindent
Next, we show the assertion by induction on $n$.
First, we assume that $g$ has a decomposition
\[
g=g_0+pg_1+\cdots+p^{n-1}g_{n-1}
\]
such that $u^r\left(\phi^r_*g_r\right) \in (f^{p^{e+n-r-1}-1})$ for every $n-1 \geq r \geq 0$.
Then $\psi^e_{n,g_0}\left(F^e_*W_n(fR)\right) \subseteq [f] \cdot W_n\omega_R$ by \cref{cl:inclusion cond}, and in particular, we may assume $g_0=0$ and $n \geq 2$.
Then $\psi^e_{n,g}=T \circ \psi^e_{n-1,u\left(\phi_*g'\right)}$ by \cref{thm:structure of dual'} (1), where
\[
g'=g_1+pg_2+\cdots +p^{n-2}g_{n-1}.
\]
If we set $g_i:=u(\phi_*g_{i+1})$, then 
\[
u\left(\phi_*g'\right)=g_0'+pg_1'+\cdots +p^{n-2}g'_{n-2}
\]
and $u^r\left(\phi^r_*g'_r\right)=u^{r+1}\left(\phi^{r+1}_*g_{r+1}\right) \in (f^{p^{e+n-2-r}-1})$ for $n-2 \geq r \geq 0$.
By the induction hypothesis, we have
\[
\psi^e_{n,g}\left(F^e_*W_n(fR)\right)=T \circ \psi^e_{n-1,u\left(\phi_*g'\right)}\left(F^e_*W_n(fR)\right) \subseteq [f] \cdot W_n\omega_R,
\]
as desired.

Next, we assume $\psi^e_{n,g}\left(F^e_*W_n(fR)\right) \subseteq [f] \cdot W_n\omega_R$.
Then for every $a \in R$, we have
\[
\psi^e_{n,g}\left(F^e_*(V^{n-1}[fa])\right)=T^{n-1} \circ \lambda_1 \circ u^{e+n-1}\left(\phi^{e+n-1}_*(gfa)\right) \in [f] \cdot W_n\omega_R.
\]
Thus, we have $u^{e+n-1}(\phi^{e+n-1}_*(gfA)) \in (f,p)$ by \cref{lem:including condition} (1) and (3).
Therefore, it follows from \cref{ideal-contained} that $g \in (f^{p^{e+n-1}-1},p)$.
Thus, we obtain a decomposition $g=g_0+ph$ such that $g_0 \in (f^{p^{e+n-1}-1})$.
Since $\psi^e_{n,g_0}\left(F^e_*W_n(fR)\right) \subseteq [f] \cdot W_n\omega_R$ by \cref{cl:inclusion cond}, we may assume $g_0=0$.
Then we have
\[
\psi^e_{n,g}\left(F^e_*W_n(fR)\right)
=
\psi^e_{n,ph}\left(F^e_*W_n(fR)\right)
=
T \circ \psi^e_{n-1,u(\phi_*h)}\left(F^e_*W_n(fR)\right)
\]
by \cref{thm:structure of dual'}(1).
By assumption, this is contained in $[f] \cdot W_n\omega_R$.
Therefore, applying \cref{lem:including condition}(1), we obtain
\[
\psi^e_{n-1,u(\phi_*h)}\left(F^e_*W_n(fR)\right)
\subseteq
[f] \cdot W_{n-1}\omega_R.
\]
By the induction hypothesis, we have a decomposition
\[
u\left(\phi_*h\right)=g'_0+pg'_1+\cdots +p^{n-2}g'_{n-2}
\]
such that $u^r\left(\phi^r_*g_r'\right) \in (f^{p^{e+n-r-2}-1})$ for $n-2 \geq r \geq 0$.
Since $u$ is surjective, there exists $g_r$ such that $u\left(\phi_*g_{r}\right)=g_{r-1}'$ for every $n-1 \geq r \geq 1$ and
\[
g=pg_1+p^2g_2+\cdots +p^{n-1}g_{n-1}.
\]
In particular, $u^r\left(\phi^r_*g_r\right)=u^{r-1}\left(\phi^{r-1}_*g_{r-1}'\right) \in (f^{p^{e+n-r-1}-1})$ for $n-1 \geq r \geq 1$, as desired.
\end{proof}

\begin{theorem}\label{thm:structure of sigma'}\textup{(\cref{intro:Fedder-qf^es})}
We use \cref{notation:fedder}.
Let $n, e \geq 1$ be integers and $f \in A$ such that $f,p$ is a regular sequence.
Then, for $g \in A$, the homomorphism $\psi^e_{n,g}$ satisfies conditions (1),(2) in \cref{prop:qF^esplit-adj-inv-adj}  if and only if
\begin{itemize}
    \item[(D1)] $u^{e+r-1}\left(\phi^{e+r-1}_*g\right) \in (p^r)$ for $n-1 \geq r \geq 1$, and 
    \item[(D2)] $g$ admits a decomposition
    \[
    g\equiv g_0+pg_1+\cdots +p^{n-1}g_{n-1} \pmod{p^n}
    \]
    such that $u^r\left(\phi^r_*g_r\right) \in (f^{p^{e+n-r-1}-1},p^n)$ for $n-1 \geq r \geq 0$.
\end{itemize}
Moreover, $R/fR$ is $n$-quasi-$F^e$-split if and only if there exists  $g \in A$ satisfying conditions \cond{D1}, \cond{D2} and condition 
\begin{itemize}
    \item[(D3)] $u^{e+n-1}\left(\phi^{e+n-1}_*g\right) \notin (\m,p^n)$.
\end{itemize}
\end{theorem}

\begin{proof}
By Proposition \ref{prop:pushout condition} and  \ref{prop:including condition 2}, the homomorphism $\psi^e_{n,g}$ satisfies conditions (1),(2) in \cref{prop:qF^esplit-adj-inv-adj}  if and only if conditions \cond{D1},\cond{D2} are satisfied.

Next, we prove the remaining assertion.
First, we assume $R/fR$ is $n$-quasi-$F^e$-split, so there exists $\psi \in (F^e_*W_n(R))^*$ such that $\psi$ satisfies conditions (1)--(3) in \cref{prop:qF^esplit-adj-inv-adj}.
By \cref{thm:structure of dual'}, there exists $g \in A$ such that $\psi=\psi^e_{n,g}$, thus $g$ satisfies conditions \cond{D1},\cond{D2}.
Furthermore, we have
\begin{align*}
    \psi^e_{n,g}\left(F^e_*1\right)
    &=\lambda_n \circ u^e\left(\phi^{e+n-1}_*g\right) = T \circ \lambda_{n-1} \circ u\left(\phi^{n-1}_*(u^e\left(\phi^e_*g\right)/p)\right) \\
    &= \cdots = T^{n-1} \circ \lambda_1 \circ u\left(\phi_*(u^{e+n-2}\left(\phi^{e+n-2}_*g\right)/p^{n-1})\right)
\end{align*}
by \cref{thm:F-inverse action'} \cond{C2} and condition \cond{D1}.
By condition (3) in \cref{prop:qF^esplit-adj-inv-adj}, we have 
\[
\lambda_1 \circ u\left(\phi_*(u^{e+n-2}\left(\phi^{e+n-2}_*g\right)/p^{n-1})\right) \notin \m \cdot \omega_R
\]
Since $\lambda_1 \colon R \to \omega_R$ is an isomorphism,  the element $u^{e+n-1}\left(\phi^{e+n-1}_*g\right)$ is not contained in $(\m,p^n)$, as desired.

Next, we prove the ``if'' part, thus we assume that there exists $g \in A$ satisfying conditions \cond{D1}--\cond{D3}.
By the argument in the proof of the ``only if'' part, we have
\[
\psi^e_{n,g}\left(F^e_*1\right)=T^{n-1} \circ \lambda_1 \circ u\left(\phi_*(u^{e+n-2}\left(\phi^{e+n-2}_*g\right)/p^{n-1})\right).
\]
By condition \cond{D3}, we have
\[
\lambda_1 \circ u\left(\phi_*(u^{e+n-2}\left(\phi^{e+n-2}_*g\right)/p^{n-1})\right) \notin \m \cdot \omega_R.
\]
Therefore, the homomorphism $\psi$ satisfies the condition (3) in \cref{prop:qF^esplit-adj-inv-adj}.
By \cref{prop:qF^esplit-adj-inv-adj}, the ring $R/fR$ is $n$-quasi-$F^e$-split, as desired.
\end{proof}

\begin{remark}\label{remark:D3-and-D3'}
We may replace condition (D3) in \cref{thm:structure of sigma'} by the following condition:
\begin{itemize}
    \item[(D3')] $u^{e+n-2}\left(\phi^{e+n-2}_*g\right) \notin (\m^{[p]},p^n)$,
\end{itemize}
where $\m^{[p]}$ is defined in \cref{notation:fedder}.

Assume that $g \in A$ satisfies conditions (D1) and (D2), and set
\[
g' :=
\frac{u^{e+n-2}(\phi^{e+n-2}_*g)}{p^{\,n-1}}
\in A.
\]
First, suppose that $g$ satisfies condition (D3).
Then $u(\phi_*g') \notin (\m,p)$, thus $g' \notin (\phi(\m),p)=(\m^{[p]},p)$.
Hence we have
\[
u^{e+n-2}\left(\phi^{e+n-2}_*g\right)
=
p^{n-1}g'
\notin
(\m^{[p]},p^n),
\]
which shows that condition (D3') holds for $g$.

Conversely, suppose that $g$ satisfies condition (D3').
Then $g' \notin (\m^{[p]},p)=(\phi(\m),p)$.
By \cref{ideal-contained}, there exists $a \in A$ such that
\[
u(\phi_*ag') \notin (\m,p).
\]
Therefore,
\[
u^{e+n-1}\left(
\phi^{e+n-1}_*\bigl(\phi^{e+n-2}(a)g\bigr)
\right)
=
p^{n-1}u(\phi_*(ag'))
\notin
(\m,p^n).
\]
Moreover, since $g$ satisfies conditions \cond{D1} and \cond{D2},
the same holds for $\phi^{e+n-2}(a)g$.
Thus the element $\phi^{e+n-2}(a)g$ satisfies conditions (D1)--(D3),
as desired.
\end{remark}

Next, as an application of \Cref{thm:structure of sigma'}, we derive several useful criteria that are well suited for explicit verification in concrete situations (see Corollary~5.11).  
For this purpose, we begin with the following lemma.

\begin{lemma}\label{e=1 case corr}
We use \cref{notation:fedder}.
Let $n \geq 1$ be an integer and $f,g \in A$ such that $f,p$ is a regular sequence.
We assume that $g$ satisfies condition \cond{D1} in \cref{thm:structure of sigma'} for $e=1$.
Let $g_r:=u^r\left(\phi^r_*g\right)/p^r$ for $n-1 \geq r \geq 0$.
Then $\psi^1_{n,g}$ corresponds to $\psi_{(g_{n-1},g_{n-2},\ldots,g_0)}$ by the isomorphism
\[
\Hom_{W_n(R)}(Q^1_{R,n},W_n\omega_R) \simeq \Hom_{W_n(R)}(Q^1_{R,n},R),
\]
where $\psi_{(g_{n-1},g_{n-2},\ldots,g_0)}$ is as in \cite{KTY}*{Lemma~4.3}.
\end{lemma}

\begin{proof}
Let 
\[
\psi:=\psi_{(g_{n-1},g_{n-2},\ldots,g_0)}.
\]
For  $a \in R$ and $n-1 \geq r \geq 0$,  we have
\begin{align*}
    \psi^1_{n,g}\left(F_*(V^r[a])\right)
    &\overset{}{=}p^r \lambda_{n} \circ u\left(\phi^{n}_*(ga^{p^{n-r-1}})\right) \\
    &\eqtag{1} \sum_{s=0}^{n-r-1} p^{r+s} \lambda_{n} \circ u\left(\phi^n_*(g\phi^{n-r-1-s}\left(\Delta_{s}(a)\right))\right) \\ 
    &\eqtag{2} \sum_{s=0}^{n-r-1}  T^{n-1} \circ \lambda_{1} \circ u^{r+s+1}\left(\phi^{r+s+1}_*(g_{n-r-s-1}\Delta_{s}(a))\right)  \\
    &\eqtag{3} T^{n-1} \circ \lambda_1 \circ \psi\left(F_*(V^r[a])\right), 
\end{align*}
where $\Delta_s$ is defined in  \cite{KTY}*{Section~3}, $(\star_1)$ follows from (3.1) in the proof of \cite{Yoshikawa25}*{Theorem~3.6},
$(\star_2)$ follows from the following computation
\begin{align*}
    p^{r+s} \lambda_{n} \circ u\left(\phi^n_*(g\phi^{n-r-1-s}\left(\Delta_{s}(a)\right))\right)
    &= p^{r+s}T \circ \lambda_{n-1} \circ u\left(\phi^{n-1}_*(g_1\phi^{n-r-2-s}\left(\Delta_s(a)\right))\right) \\
    &= \cdots \\
    &= p^{r+s}T^{n-r-1-s} \circ \lambda_{r+1+s} \circ u\left(\phi^{r+1+s}_*(g_{n-r-s-1}\Delta_s(a))\right) \\
    &= T^{n-1} \circ \lambda_{1} \circ u^{r+s+1}\left(\phi^{r+s+1}_*(g_{n-r-s-1}\Delta_s(a))\right),
\end{align*}
and $(\star_3)$ follows from \cite{KTY}*{Lemma~4.3}, as desired.
\end{proof}

\begin{lemma}\label{lem:e=1_case}
We use \cref{notation:fedder}.
Let $n \geq 1$ be an integer and $f,g \in A$ such that $f,p$ is a regular sequence.
We assume $g$ satisfies condition \cond{D1}--\cond{D3} in \cref{thm:structure of sigma'} for $e=1$.
Then there exists $h_0,\ldots,h_{n-1} \in A$ such that 
\begin{enumerate}
    \item $h_0 \in (f^{p-1},p)$ and $g-f^{p^n-p}h_0 \in (p)$,
    \item $u\left(\phi_*h_i\right) \in (p)$ and $h_{i+1}-u\left(\phi_*(\Delta_1(f^{p-1})h_i)\right) \in (f^{p-1},p)$ for every $n-2 \geq i \geq 0$, and
    \item $h_{n-1} \notin (\m^{[p]},p)$.
\end{enumerate}
\end{lemma}

\begin{proof}
We take $g_0,\ldots,g_{n-1}$ as in \cref{e=1 case corr}.
By \cite{KTY}*{Theorem~4.8 and Lemma~4.9}, there exists $h_0,\ldots,h_{n-1}$ such that $h_0 \in (f^{p-1},p)$, $f^{p^i-p}h_{n-1-i}=g_i$ for every integer $n-1 \geq i \geq 0$, $u\left(\phi_*h_{n-1-i}\right) \in (p)$ for $i \geq 2$, and
\[
h_{n-1-i}-u\left(\phi_*\Delta_1(f^{p-1}h_{n-2-i})\right) \in (f^{p-1},p)
\]
for $n-1 \geq i \geq 0$.
Furthermore, by the proof of \cite{KTY}*{Lemma~4.10}, we have $h_{n-1} \notin (\m^{[p]},p)$, as desired.
\end{proof}

\begin{corollary}\label{nec-suff-cond}\textup{(\cref{intro:nec-suff-cond})}
We use \cref{notation:fedder}.
Let $n, e \geq 1$ be integers and $f \in A$ such that $f,p$ is a regular sequence.
\begin{enumerate}
    \item If there exists $g \in f^{p^{e+n-1}-1}A$ satisfying conditions \cond{D1} and \cond{D3}, then $R/fR$ is $n$-quasi-$F^e$-split.
    \item We define a sequence of ideals $\{I^e_n\}$ of $A$ inductively as follows.  
    The ideal $I^e_1$ is defined by  
    \[
    I^e_1 := f^{p-1} u^{e-1}\left(\phi^{e-1}_*(f^{p^{e-1} - 1} A)\right),
    \]  
    and for each $n \geq 1$, we define  
    \[
    I^e_{n+1} := u\left(\phi_*( \Delta_1(f^{p-1}) ( I^e_n \cap u^{-1}(pA) )) \right) + f^{p-1}A.
    \]
    If $R/fR$ is $n$-quasi-$F^e$-split, then $I^e_n \nsubseteq (\m^{[p]},p)$.
    \item We define a sequence of ideals $\{I'_n\}$ of $A$ inductively as follows.  
    The ideal $I'_1$ is defined by  
    \[
    I'_1 := f^{p-1} \mathfrak{m} ,
    \]  
    and for each $n \geq 1$, we define  
    \[
    I'_{n+1} := u\left(\phi_*( \Delta_1(f^{p-1}) ( I'_n \cap u^{-1}(pA) )) \right) + f^{p-1}A.
    \]
    If $R/fR$ is not $F$-pure and $I'_n \subseteq (\m^{[p]},p)$, then $R/fR$ is not $n$-quasi-$F^2$-split.
\end{enumerate}
\end{corollary}

\begin{proof}
Assertion (1) follows from \cref{thm:structure of sigma'}.
We prove Assertion (2).
Since $R/fR$ is $n$-quasi-$F^e$-split, there exists $g \in A/p^n$ satisfying conditions \cond{D1}--\cond{D3} in \cref{thm:structure of sigma'}.
By condition \cond{D2}, there exists $g_0 \in (f^{p^{e+n-1}-1})$ such that $g \equiv g_0 \mod (p)$.
Set
\[
\wt{g}
:=
u^{e-1}\left(\phi^{e-1}_*g\right)
\in A/p^n.
\]
Iterating \cref{prop:psi^e_n is module hom}~(4), we obtain
\[
\psi^1_{n,\wt{g}}(F_*\alpha)
=
\psi^e_{n,g}\left(F^e_*(F^{e-1}\alpha)\right)
\]
for every $\alpha\in W_n(R)$.
It follows that $\psi^1_{n,\wt{g}}$ satisfies conditions
\textup{(1)}--\textup{(3)} in
\cref{prop:qF^esplit-adj-inv-adj}.
Indeed, condition \textup{(1)} follows from the defining pushout
diagrams, condition \textup{(2)} follows from
\[
F^{e-1}\left(W_n(fR)\right)\subseteq W_n(fR),
\]
and condition \textup{(3)} follows by taking $\alpha=1$.
Therefore, by \cref{thm:structure of sigma'},
the element $\wt{g}$ satisfies conditions
\cond{D1}--\cond{D3} for $e=1$.
Hence, by \cref{lem:e=1_case}, there exist
$h_0,\ldots,h_{n-1}$ satisfying conditions
\textup{(1)}--\textup{(3)} for $\wt{g}$.
Thus, the element $h_0$ is contained in $I^e_1+(p)$.
Furthermore, we obtain that $h_i \in I^e_{i+1}+(p)$ inductively for $n-1 \geq i \geq 0$ by condition (2) in \cref{lem:e=1_case}.
By condition (3) in \cref{lem:e=1_case}, the ideal $I^e_n$ is not contained in $(\m^{[p]},p)$, as desired.

Finally, we prove Assertion (3).
Since $R/fR$ is not $F$-pure, we have $u\left(\phi_*(f^{p-1}A)\right)R \subseteq \m$, and in particular, we have $I^2_1 \subseteq I'_1$.
Furthermore, we obtain $I^2_m \subseteq I'_m$ for every $m \geq 1$ inductively.
Therefore,  Assertion (3) follows from Assertion (2).
\end{proof}

Next, we state a Fedder-type criterion for quasi-$F$-regularity.

\begin{theorem}\label{fedder-qFr}\textup{(\cref{intro:Fedder-qFr})}
We use \cref{notation:fedder}.
Let $n, e \geq 1$ be integers, $f \in A$ be an element such that $f,p$ is a regular sequence, $t \in \tau(R/fR)$ and $c \in A$ such that 
\[
(A \to R/fR)(c) \in (t^4) \cap (R/fR)^\circ.
\]
Then $R/fR$ is $n$-quasi-$F$-regular if and only if there exists $g \in A$ such that $g$ satisfies condition \cond{D2} and $gc^{p^n-1}$ satisfies conditions \cond{D1},\cond{D3} in \cref{thm:structure of sigma'}.
\end{theorem}

\begin{proof}
First, we prove the ``only if'' part, thus we assume that $R/fR$ is $n$-quasi-$F$-regular.
By \cref{qFr-inv-adj} and \cref{thm:structure of dual'}, there exist $g \in A/p^n$ and $e \geq 1$ such that $\psi^e_{n,g}$ satisfies conditions (1)--(3) in \cref{qFr-inv-adj}.
By condition (2), the element $g$ satisfies condition \cond{D2}.
Furthermore, since $\psi^e_{n,g}$ satisfies conditions (1) and (3), the homomorphism $\psi^e_{n,g}([c]\cdot-)$ satisfies conditions \cond{D1} and \cond{D3}.
Since we have
\begin{align*}
    \psi^e_{n,g}\left(F^e_*([c]\alpha)\right)
    &=\lambda_n \circ u^e\left(\phi^{e+n-1}_*(g\varphi_{n-1}([c]\alpha))\right) \\
    &=\lambda_n \circ u^e\left(\phi^{e+n-1}_*(gc^{p^{n-1}}\varphi_{n-1}(\alpha))\right)
    =\psi^e_{n,gc^{p^{n-1}}}\left(F^e_*\alpha\right),
\end{align*}
the homomorphism $\psi^e_{n,g}\left(F^e_*([c]\cdot-)\right)$ coincides with $\psi^e_{n,gc^{p^{n-1}}}$.
Therefore, we obtain that $gc^{p^{n-1}}$ satisfies conditions \cond{D1} and \cond{D2}, as desired.

Next, we prove the ``if'' part, thus we assume that $g$ satisfies condition \cond{D2} and $gc^{p^{n-1}}$ satisfies conditions \cond{D1} and \cond{D3}.
By the argument as above, the homomorphism $\psi^e_{n,g}$ satisfies conditions (1)--(3) in \cref{qFr-inv-adj}.
Therefore, the ring $R/fR$ is $n$-quasi-$F$-regular by \cref{qFr-inv-adj}.
\end{proof}

\begin{corollary}\label{qFr-suff-cond}
We use \cref{notation:fedder}.
Let $n \geq 1$, $f,c \in A$, $t \in \tau(R/fR)$ such that $f,p$ is a regular sequence, and
\[
(A \to R/fR)(c) \in (t^4) \cap (R/fR)^\circ.
\]
If there exists an integer $e \geq 1$ and $g \in (f^{p^{e+n-1}-1},p^n)A$ such that $gc^{p^n-1}$ satisfies conditions \cond{D1},\cond{D3}, then $R/fR$ is $n$-quasi-$F$-regular.
\end{corollary}

\begin{proof}
The assertion follows from \cref{fedder-qFr}.
\end{proof}

\begin{example}\label{example}
\begin{enumerate}
    \item Let $p=2$, $A:=\Z_{(p)}[[x,y,z]]$, $R:=A/p$, and $f:=z^2+x^3+y^2z$.
    Then $R/fR$ is $2$-quasi-$F$-regular.
    Let $\phi \colon A \to A$ be a Frobenius lift defined by $\phi(g(x,y,z))=g(x^p,y^p,z^p)$.
    Then $\{\phi_*(x^iy^jz^l) \mid p-1 \geq i,j,l \geq 0 \}$ is a basis of $\phi_*A$ over $A$.
    We consider the dual basis, and the homomorphism $\phi_*A \to A$ in the dual basis corresponding to $\phi_*(x^{p-1}y^{p-1}z^{p-1})$ is denoted by $u$, then $u$ is a generator of $\Hom_{A}(\phi_*A,A)$ as a $\phi_*A$-module. 
    Since $x^2y$ is contained in the Jacobian ideal of $R/fR$, we have $x^2y \in \tau(R/fR)$ by \cite{HH02}*{Theorem~3.4}.
    Moreover, a direct computation gives
\[
u(\phi_*(x^3yf))=xy,\qquad
u(\phi_*(xyzf))=z,\qquad
u(\phi_*(yzf))=x.
\]
    Since we have
    \[
u(\phi_*(f^2A)) \equiv u(\phi_*(\phi(f)A))=f\cdot u(\phi_*A)=fA,
    \]
    the map $u(F_*(f \cdot-))$ induces an $R/f$-module homomorphism $F_*R/f \to R/f$, thus 
    the element $x$ is contained in $\tau(R/fR)$ by the definition of $\tau(R/f)$ (see \cite{ST11}*{Definition~3.12}).
Now set $e=6$, $c=x^4$, and $g:=x^7y^{15}zf^{127}$.
We show $u^6(\phi^6_*(c^3g)) \in (p)$ and $u^7(\phi^7_*(c^3g)) \notin (\m,p^2)$.

First note that
\begin{align*}
    c^3g 
    &= x^{12}\cdot x^7y^{15}z\cdot f^{127}
= x^{19}y^{15}z (z^2+x^3+y^2z)^{127} \\
&=\sum_{\substack{a+b+c=127\\ a,b,c\ge0}}
\binom{127}{a,b,c}x^{19+3b}y^{15+2c}z^{1+2a+c}.
\end{align*}
In order to compute $u^6(\phi^6_*(c^3g))$, we solve the congruences
\[
19+3b\equiv 63 \pmod{64},\qquad
15+2c\equiv 63 \pmod{64},\qquad
1+2a+c\equiv 63 \pmod{64},
\]
under the constraint $a+b+c=127$.
By direct computation, we get
\[
(a,b,c)=(35,36,56).
\]
Hence it follows that
\[
u^{6}(\phi^{6}_*(c^3g))
=\binom{127}{35,36,56} xyz.
\]
Since the $2$-adic valuation of $\binom{127}{35,36,56}$ is $1$, we have $u^6(\phi^6_*(c^3g)) \in (p)$ and 
\[
u^7(\phi^7_*(c^3g)) \equiv 2 \pmod{4},
\]
and in particular, we have $u^7(\phi^7_*(c^3g)) \notin (\m,p^2)$.
    Therefore, the ring $R/fR$ is $2$-quasi-$F$-regular by \cref{qFr-suff-cond}.
    \item
Let $p=2$, $A:=\Z_{(p)}[[x,y,z,w,v]]$, $R:=A/(p)$, and
\[
f:=zwv^2+y^3w+x^3z \in R.
\]
Let $\phi$ be defined as in (1).
Then $\{\phi_*(x^iy^jz^kw^\ell v^m)\mid 0\le i,j,k,\ell,m\le1\}$ is an $A$-basis of $\phi_*A$,
and $u\in\Hom_A(\phi_*A,A)$ denotes the dual basis element corresponding to
$\phi_*(xyzwv)$.
Since $x^2z$ belongs to the Jacobian ideal of $R/(f)$, we have $x^2z\in\tau(R/(f))$.
Moreover, a direct computation shows that
\[
u(\phi_*(x^3zvf))=xy,\qquad
u(\phi_*(xyvf))=v.
\]
Hence $v\in\tau(R/(f))$.

Now set $e=5$, $c=v^4$, and
\[
g:=x^6y^3z^{20}w^{19}v^3f^{63}.
\]
We claim that $u^{5}(\phi^{5}_*(c^3g))\in(2)$ and $u^6(\phi^6_*(c^3g)) \notin (\m,p^2)$.

By a computation as in (1), we obtain
\[
u^{5}(\phi^{5}_*(c^3g))=\binom{63}{24,20,19}xyzwv
\]
Since the $2$-adic order of $\binom{63}{24,20,19}$ is one, we obtain that $R/(f)$ is $2$-quasi-$F$-regular by \cref{qFr-suff-cond} and an argument as in (1).

Furthermore, we set $S:=\F_{(2)}[[x,y,z,w,v,s]]$ and
\[
f':=xys^2 +zwv^2 +y^3w+x^3z. 
\]
Then $S/(f')$ is not quasi-$F$-split by \cite{KTY}*{Example~7.7} and $S/(f',s) \simeq R/fR $ is quasi-$F$-regular as shown above.
In conclusion, we have therefore obtained a counterexample to inversion of adjunction for
quasi-$F$-regularity.
\end{enumerate}
\end{example}

\section{Quasi-$F$-splitting threshold}

\subsection{Preliminary}
\begin{definition}[cf.~\cite{TWY24}*{Proposition~3.19}]
Let $R$ be an $F$-finite normal ring of characteristic $p$, and let $D$ be an effective $\R$-Weil divisor on $X := \Spec R$.
We say that $(X,D)$ is \emph{$n$-quasi-$F^e$-split} if $\rdown{D}=0$ and there exists
\[
\varphi \in \Hom_{W_n\cO_X}\bigl(F^e_*W_n\cO_X(p^eD),W_n\omega_X(-K_X)\bigr)
\]
such that the following diagram commutes:
\begin{equation}\label{eq:def:p-qFs}
\begin{tikzcd}
W_n\cO_X(D) \arrow[rrr,"F^e"] \arrow[d,"Res"'] 
& & & 
F^e_*W_n\cO_X(p^eD) \arrow[d,"\varphi"] \\
\cO_X \arrow[r,equal] 
& \omega_X(-K_X)  
\arrow[rr,"T^{n-1}(-K_X)"]  
& & 
W_n\omega_X(-K_X).
\end{tikzcd}
\end{equation}
For the definitions of $W_n\cO_X(D)$, $W_n\cO_X(p^eD)$, and $W_n\omega_X(-K_X)$, see
\cite{TWY24}*{Section~2.2}.
\end{definition}

\begin{definition}\label{defn:threshold}
Let $R$ be an $F$-finite normal ring of characteristic $p$, and let $D$ be an effective $\R$-Weil divisor on $X := \Spec R$.
We define the following thresholds:
\begin{itemize}
    \item The \emph{quasi-$F$-pure threshold} of $D$ with respect to $X$ is
    \[
    \qfpt(X;D):=\sup\{a \in \R_{\geq 0} \mid (X,aD)\ \text{is quasi-$F^e$-split for all $e \geq 1$}\},
    \]
    \item The \emph{bounded quasi-$F$-pure threshold} of $D$ with respect to $X$ is
    \[
    \bfpt(X;D):=\sup\{ a \in \R_{\geq 0} \mid \exists\, n \ \text{such that}\ (X,aD)\ \text{is $n$-quasi-$F^e$-split for all $e \geq 1$}\}.
    \]
\end{itemize}
\end{definition}

\begin{proposition}\label{prop:easy property}
Let $R$ be an $F$-finite normal ring of characteristic $p$, and let $D$ be an effective $\R$-Weil divisor on $X := \Spec R$.
Then
\[
\fpt(X;D) \leq \bfpt(X;D) \leq \qfpt(X;D).
\]
\end{proposition}

\begin{proof}
Let $a \in \R_{\geq 0}$.
First, assume $a < \fpt(X;D)$.
Then $(X,aD)$ is $F$-pure, so for all $e \geq 1$, the map
\[
\cO_X \to F^e_*\cO_X\bigl(\lceil (p^e-1)aD \rceil\bigr)
\]
splits.
Therefore, $(X,\tfrac{p^e-1}{p^e}aD)$ is 1-quasi-$F^e$-split.
In particular, for any $0 < \epsilon < 1$, $(X,\epsilon aD)$ is $1$-quasi-$F^e$-split for all $e$, so we have $a \leq \bfpt(X;D)$.
This proves the first inequality.
The remaining inequalities are immediate from the definitions.
\end{proof}

\begin{proposition}\label{prop:bdd and iterated}
Let $R$ be an $F$-finite normal ring of characteristic $p$, and let $D$ be an effective $\R$-Weil divisor on $X := \Spec R$.
Assume that $R$ is $F$-pure and Cohen--Macaulay.
\begin{enumerate}
    \item If $p^e(K_X+D)$ is Cartier and $(X,D)$ is $n$-quasi-$F^e$-split, then $(X,D)$ is $n$-quasi-$F^{e+1}$-split.
    \item If $R$ is regular, then $\qfpt(X;D)=\bfpt(X;D)$.
\end{enumerate}
\end{proposition}

\begin{proof}
For (1), it suffices to prove the surjectivity of the map
\[
\Hom_{W_n\cO_X}(F^{e+1}_*W_n\cO_X(p^{e+1}D),W_n\omega_X(-K_X))
\to
\Hom_{W_n\cO_X}(F^{e}_*W_n\cO_X(p^{e}D),W_n\omega_X(-K_X)).
\]
By the $S_2$ condition for $W_n\omega_X(-K_X)$ (see \cite{TWY24}*{Section~2.2}) and restricting the regular locus of $X$, we obtain
\[
\Hom_{W_n\cO_X}(F^{e'}_*W_n\cO_X(p^{e'}(K_X+D)),W_n\omega_X) \simeq \Hom_{W_n\cO_X}(F^{e'}_*W_n\cO_X(p^{e'}D),W_n\omega_X(-K_X))
\]
for every $e' \geq 1$, thus it suffices to prove the surjectivity of the map
\[
\Hom_{W_n\cO_X}(F^{e+1}_*W_n\cO_X(p^{e+1}(K_X+D)),W_n\omega_X)
\to
\Hom_{W_n\cO_X}(F^{e}_*W_n\cO_X(p^{e}(K_X+D)),W_n\omega_X).
\]
Consider the exact sequence
\[
0 \to W_n\cO_X \to F_*W_n\cO_X \to B_{X,n} \to 0.
\]
Since $p^e(K_X+D)$ is Cartier, the cokernel of
\[
F^e_*W_n\cO_X(p^e(K_X+D))
\to
F^{e+1}_*W_n\cO_X(p^{e+1}(K_X+D))
\]
is $F^e_*B_{X,n}(p^e(K_X+D))$.
Because $X$ is $F$-pure and Cohen--Macaulay, $B_{X,n}$ is Cohen--Macaulay by the proof of \cite{KTTWYY1}*{Proposition~3.9}\footnote{ \cite{KTTWYY1}*{Proposition~3.9} assumes that $X$ is divisorially Cohen--Macaulay.
However, in the proof, only the Cohen--Macaulayness of $\cO_X(p^j\Delta)$ for every $j \geq 0$ is used.
In our situation, we apply the same argument in the case $\Delta=0$.
Therefore, it suffices to assume that $X$ is Cohen--Macaulay.}.
In particular, so is $F^e_*B_{X,n}(p^e(K_X+D))$.
Therefore,
\[
\Ext^1_{W_n\cO_X}(F^e_*B_{X,n}(p^e(K_X+D)),W_n\omega_X)=0,
\]
which yields the surjection
\[
\Hom_{W_n\cO_X}(F^{e+1}_*W_n\cO_X(p^{e+1}(K_X+D)),W_n\omega_X)
\to
\Hom_{W_n\cO_X}(F^{e}_*W_n\cO_X(p^{e}(K_X+D)),W_n\omega_X).
\]
Hence the desired surjectivity follows.

For (2), by Proposition~\ref{prop:easy property}, it is enough to show that $\bfpt(X;D) \geq \qfpt(X;D)$.
Take $a < \qfpt(X;D)$ and set
\[
D_e := \frac{\lfloor p^e aD \rfloor}{p^e}.
\]
Since $X$ is regular, $p^eD_e$ is Cartier.
Moreover, for all $e$, $(X,D_e)$ is $n_e$-quasi-$F^e$-split for some $n_e$.
By (1), $(X,D_e)$ is $n_e$-quasi-$F^{e'}$-split for all $e'$.
Thus, it remains to show that for any $a' < a$, there exists $e$ such that $a'D \leq D_e$.
For this, we may assume $D = bE$ for some $b \in \R_{\geq 0}$ and some prime divisor $E$.
Then
\[
D_e = \Bigl(\frac{\lfloor p^e ab \rfloor}{p^e}\Bigr) E.
\]
Since $\lim_{e \to \infty} \frac{\lfloor p^e ab \rfloor}{p^e} = ab$, the claim follows.
\end{proof}

\begin{example}
Let $R$ be an $F$-finite regular ring of characteristic $p$, and let $S$ be a prime divisor on $X := \Spec R$.
If $S$ is quasi-$F^e$-split for all $e$, then
\[
\qfpt(X;S) = \bfpt(X;S) = 1.
\]

To prove this, we may assume that $S = \mathrm{div}(f)$ for some $f \in R$.
Since $R/f$ is $n_e$-quasi-$F^e$-split for some $n_e$, by the proof of \cref{prop:qF^esplit-adj-inv-adj}, there exists  a $W_{n_e}(R)$-module homomorphism
\[
\varphi \colon F^e_*W_{n_e}(R) \to W_{n_e}\omega_R
\]
such that $\varphi(F^e_*W_{n_e}(fR)) \subseteq [f]\cdot W_{n_e}\omega_R$ and the following diagram commutes:
\[
\begin{tikzcd}
W_{n_e}(R) \arrow[r] \arrow[d] & F^e_*W_{n_e}(R) \arrow[d,"\varphi"] \\
R \arrow[r] & W_{n_e}\omega_R.
\end{tikzcd}
\]
In particular, the map $\varphi$ induces a homomorphism
\[
\varphi' \colon [f]^{-1}F^e_*W_{n_e}(fR) \to W_{n_e}\omega_R,
\]
and we obtain a commutative diagram
\[
\begin{tikzcd}
{[f]}^{-1}W_{n_e}(fR) \arrow[r] \arrow[d] 
& {[f]}^{-1}F^e_*W_{n_e}(fR) \arrow[d,"\varphi'"] \\
R \arrow[r] & W_{n_e}\omega_R.
\end{tikzcd}
\]
Since
\[
[f]^{-1}W_{n_e}(fR)=W_{n_e}\cO_X(a_eS),
\qquad
[f]^{-1}F^e_*W_{n_e}(fR)=F^e_*W_{n_e}\cO_X(p^ea_eS),
\]
where
\[
a_e:=1-\frac{1}{p^{e+n_e-1}},
\]
it follows that the pair $(X,a_eS)$ is $n_e$-quasi-$F^e$-split for every $e$.

For sufficiently large $e$, we have $a_e \to 1$, and hence this implies that
\[
\qfpt(X;S)=1.
\]
Therefore, by \cref{prop:bdd and iterated}, we conclude that
\[
\bfpt(X;S)=1
\]
as well.
\end{example}

\subsection{Criterion of quasi-$F^e$-splitting for pairs}

\begin{notation}\label{conv:criterion}
We use \cref{notation:fedder}.
Let $\pi \colon A \to R$ be the natural surjection and set $X := \Spec R$.
Let $D$ be an effective $\Q$-divisor on $\Spec A$ such that $p^eD$ is Cartier for some $e \in \Z_{\geq 1}$.
Fix such an $e$ and choose $f \in A$ with $\mathrm{div}(f) = p^eD$.
By abuse of notation, we also write $\pi(f)$ as $f$, and denote by $D$ the restriction $D|_{\Spec R}$.
We assume $\pi(f) \neq 0$.
Note that if $(X,D)$ is quasi-$F^e$-split, then by Proposition~\ref{prop:bdd and iterated}, $(X,D)$ is quasi-$F^{e'}$-split for all $e' \geq 1$.
\end{notation}

\begin{lemma}\label{lem:first lem}
$(X,D)$ is $n$-quasi-$F^e$-split if and only if there exists 
\[
\psi \in \Hom_{W_n(R)}(F^e_*W_n(R),W_n\omega_R)
\]
such that 
\begin{enumerate}
    \item $\psi\bigl(F^e_*W_n\cO_X(p^e(1-p^e)D)\bigr) \subseteq [f]W_n\omega_R$, 
    \item $\psi$ induces a homomorphism $Q^e_{R,n} \to W_n\omega_R$, and
    \item $\psi(F^e_*1)=T^{n-1}\circ \lambda_1(1)$.
\end{enumerate}
\end{lemma}

\begin{proof}
First, assume that $(X,D)$ is $n$-quasi-$F^e$-split.
Then there exists
\[
\psi' \colon F^e_*W_n\cO_X(p^eD) \to W_n\omega_X
\]
such that $\psi'(F^e_*1) = T^{n-1}\circ \lambda_1(1)$.
Define 
\[
\psi \colon F^e_*W_n\cO_X \hookrightarrow F^e_*W_n\cO_X(p^eD) \xrightarrow{\ \psi'\ } W_n\omega_X.
\]
Then $\psi$ clearly satisfies conditions \textup{(2)} and \textup{(3)}.
Moreover,
\[
\psi\bigl(F^e_*W_n\cO_X(p^e(1-p^e)D)\bigr) 
=\psi([f]F^e_*W_n\cO_X(p^eD)) 
\subseteq [f]W_n\omega_X,
\]
so condition \textup{(1)} is also satisfied.

Conversely, assume that there exists $\psi \in \Hom_{W_n\cO_X}(F^e_*W_n\cO_X,W_n\omega_X)$ satisfying \textup{(1)}--\textup{(3)}.
Since $Q^e_{X,D,n} = Q^e_{X,n}(p^eD)$, the homomorphism $\psi$ induces
\[
\psi' \colon Q^e_{X,D,n} \to W_n\omega_X
\]
with $\psi'(F^e_*1) = T^{n-1}\circ \lambda_1(1)$ by (1)--(3).
Therefore the pair $(X,D)$ is $n$-quasi-$F^e$-split, as claimed.
\end{proof}

\begin{lemma}\label{lem:inducing criterion}
Let $\psi \in \Hom_{W_n(R)}(F^e_*W_n(R),W_n\omega_R)$. 
Then $\psi$ satisfies condition \textup{(1)} in Lemma~\ref{lem:first lem} if and only if there exists $g \in (f^{p^{n-1}},p^n)$ such that $\psi^e_{n,g} = \psi$.
\end{lemma}

\begin{proof}
First, assume that there exists $g \in (f^{p^{n-1}},p^n)$ with $\psi = \psi^e_{n,g}$.
Then for any $a \in R$ and $n-1 \geq r \geq 0$, we have
\[
\psi^e_{n,g}\bigl(\phi^{e}_*(V^r[f^{(p^e-1)p^{r}}a])\bigr)
= T^r \circ \lambda_{n-r} \circ u^{e+r}\bigl(\phi^{e+n-1}_*(g f^{(p^e-1)p^{n-1}} a^{p^{n-r-1}})\bigr)
\]
by \cref{prop:psi^e_n is module hom} (1).
Since $g f^{(p^e-1)p^{n-1}} \in (f^{p^{e+n-1}},p^n)$, we may write
\[
g f^{(p^e-1)p^{n-1}} \equiv g' f^{p^{e+n-1}} \pmod{p^n},
\]
and hence
\[
\psi^e_{n,g}\bigl(\phi^e_*(V^r[f^{(p^e-1)p^{r}}a])\bigr)
= \psi^e_{n,g' f^{p^{e+n-1}}}(\phi^e_*(V^r[a]))
= [f]\psi^e_{n,g'}(\phi^e_*(V^r[a]))
\]
by (4) of Proposition~\ref{thm:structure of dual'}.
Thus condition \textup{(1)} in Lemma~\ref{lem:first lem} holds.

Conversely, assume that $\psi$ satisfies condition \textup{(1)} in Lemma~\ref{lem:first lem}.
Take $g \in A$ with $\psi^e_{n,g} = \psi$ by \cref{thm:structure of dual'} (2).
Then for all $a \in R$, we have
\[
\psi^e_{n,g}\bigl(\phi^e_*(V^{n-1}[af^{(p^e-1)p^{n-1}}])\bigr)
= T^{n-1} \circ \lambda_1 \circ u^{e+n-1}\bigl(\phi^{e+n-1}_*(g a f^{(p^e-1)p^{n-1}})\bigr) \in [f]W_n\omega_R.
\]
By Lemma~\ref{lem:including condition} (1), it follows that
\[
u^{e+n-1}\bigl(\phi^{e+n-1}_*(g a f^{(p^e-1)p^{n-1}})\bigr) \in (f,p)
\quad \text{for all } a \in R.
\]
Therefore $g \in (f^{p^{n-1}},p)$ by \cref{ideal-contained}, and hence we may write
\[
g = g_1 + p g_2,
\]
where $g_1 \in (f^{p^{n-1}})$.
By the above argument, $\psi^e_{n,g_1}$ satisfies condition \textup{(1)} in Lemma~\ref{lem:first lem}, and in particular, so does $\psi^e_{n,pg_2}$.
Since $p = F V$ on $W_n(R)$, we have
\[
\psi^e_{n,pg_2}
\bigl(
\phi^e_*\bigl(V^{n-2}[a f^{(p^e-1)p^{n-2}} ]\bigr)
\bigr)
=
\psi^e_{n,g_2}
\bigl(
\phi^e_*\bigl(V^{n-1}[a^p f^{(p^e-1)p^{n-1}}]\bigr)
\bigr).
\]
Therefore, as above, it follows that
\[
u^{e+n-1}
\bigl(
\phi^{e+n-1}_*(g_2 a^p f^{(p^e-1)p^{n-1}} )
\bigr)
\in (f,p)
\]
for all $a \in R$.
By \cref{ideal-contained}, we have $u(\phi_*g_2) \in (f^{p^{n-2}},p)$.

Set
\[
h:=u(\phi_*g_2) \in (f^{p^{n-2}},p).
\]
Since $u$ is surjective, we may choose $c\in A$ such that $u(\phi_*c)=1$.
Define
\[
g'_2:=\phi(h)c.
\]
Since $h\in(f^{p^{n-2}},p)$, we have
\[
g'_2
\in
\bigl(\phi(f)^{p^{n-2}},p\bigr)
\subseteq
(f^{p^{n-1}},p).
\]
Moreover, we have
\[
u(\phi_*g'_2)=
u\left(\phi_*(\phi(h)c)\right)=
h\,u(\phi_*c)=
h=
u(\phi_*g_2).
\]
Therefore, by \cref{thm:structure of dual'}~(1),
\[
\begin{aligned}
\psi^e_{n,pg'_2}
&=
T\circ\psi^e_{n-1,u(\phi_*g'_2)}\\
&=
T\circ\psi^e_{n-1,u(\phi_*g_2)}\\
&=
\psi^e_{n,pg_2}.
\end{aligned}
\]
Consequently,
\[
\psi^e_{n,g_1+pg'_2}
=
\psi^e_{n,g_1+pg_2}
=
\psi.
\]
Thus, replacing $g_2$ by $g'_2$, we may assume that
\[
g_2\in(f^{p^{n-1}},p).
\]
In particular, the map $\psi$ is represented by an element of
\[
(f^{p^{n-1}},p^2).
\]
Repeating this procedure inductively, we obtain a representative
of $\psi$ contained in
\[
(f^{p^{n-1}},p^n),
\]
as desired.
\end{proof}

\begin{theorem}\label{thm:log criterion}
$(X,D)$ is $n$-quasi-$F^e$-split  if and only if there exists $g \in A$ such that
\begin{itemize}
    \item[(D1)] $u^{e+r-1}(\phi^{e+r-1}_*g) \in (p^r)$ for all $1 \leq r \leq n-1$,
    \item[(D2)] $g \in (f^{p^{n-1}},p^n)$, and
    \item[(D3)] $u^{e+n-1}(\phi^{e+n-1}_*g) \notin (\m,p^n)$.
\end{itemize}
\end{theorem}

\begin{proof}
First, assume that $(X,D)$ is $n$-quasi-$F^e$-split.
Then there exists 
\[
\psi \in \Hom_{W_{n}(R)}(F^e_*W_n(R),W_n\omega_R)
\]
satisfying conditions \textup{(1)}--\textup{(3)} in Lemma~\ref{lem:first lem}.
By Lemma~\ref{lem:inducing criterion}, there exists $g \in (f^{p^{n-1}},p^n)$ such that $\psi = \psi^e_{n,g}$.
By the proof of Theorem~\ref{thm:structure of sigma'}, the homomorphism $\psi$ satisfies conditions (D3) and (D1).

Conversely, assume that there exists $g \in A$ satisfying conditions (D1)--(D3).
By the proof of Theorem~\ref{thm:structure of sigma'}, the homomorphism $\psi = \psi^e_{n,g}$ satisfies conditions \textup{(2)} and \textup{(3)} in Lemma~\ref{lem:first lem}, after replacing $g$ by $\phi^{e+n-2}(a)g$ for some $a \in A$.
Note that such $g$ still satisfies (D2).
By Lemma~\ref{lem:inducing criterion}, $\psi^e_{n,g}$ also satisfies condition \textup{(1)} in Lemma~\ref{lem:first lem}.
Therefore, by Lemma~\ref{lem:first lem}, $(X,D)$ is $n$-quasi-$F^e$-split, as claimed.
\end{proof}

\subsection{Quasi-$F$-pure threshold for the cone of the ordinary cusp}

\begin{example}\label{cusp}
Let $k$ be a perfect field.
Let $R=k[[x,y,z]]$, $A=W(k)[[x,y,z]]$, $f=x^3+y^2z$, $X=\Spec R$ and $D=\mathrm{div}(f)$.
We prove $\qfpt(X;D)=5/6$.

First, we prove that $\qfpt(X;D) \leq 5/6$.
We take $a \in \Q_{\geq 0}$ such that $p^ea \in \Z$ and $(X,aD)$ is quasi-$F^e$-split, then we have
\[
f^{ap^{e+n-1}} \notin (\m^{[p^{e+n-1}]},p^n)
\]
for some $n$ by \cref{thm:log criterion}.
Then there exists non-negative integers $A,B$ such that $A+B=ap^{e+n-1}$ and
\[
(x^3)^A(y^2z)^B \notin (\m^{[p^{e+n-1}]}).
\]
Therefore, we have $3A < p^{e+n-1}$ and $2B< p^{e+n-1}$, thus 
\[
ap^{e+n-1}<\frac{5}{6}p^{e+n-1},
\]
as desired.

Next, we prove $\qfpt(X;D) \geq 5/6$.
We note that if $p \equiv 1 \mod 6$, then $\fpt(X;D)=5/6$ by \cite{MTW}*{Example~4.3}.
Therefore, we may assume $p \not\equiv 1 \mod 6$ by \cref{prop:easy property}.
\begin{claim}\label{claim: cusp 1}
We define $t_e$ by
\begin{equation*}
  t_e:=
  \begin{cases}
    \frac{5}{3}\frac{p^{e-1}-1}{p^e}                & \text{if $p=2$,} \\
    \frac{5}{2}\frac{p^{e-1}-1}{p^e}                 & \text{if $p=3$,} \\
    \frac{5}{6}\frac{p^e-1}{p^e}                    & \text{if $p \equiv 5 \mod 6$.}
  \end{cases}
\end{equation*}
Then we have
\[
f^{t_ep^e} \notin \m^{[p^e]}
\]
if $p=2$ and $e$ is odd, or $p \neq 2$ and $e$ is even.
\end{claim}
\begin{claimproof}
It is enough to find non-negative integers $A_0,B_0$ such that 
\[
3A_0 \leq p^e-1,\ 2B_0 \leq p^e-1\ \text{and}\ A_0+B_0=t_ep^e.
\]
The following $A,B$ satisfy the conditions.
\begin{equation*}
  (A_0,B_0)=
  \begin{cases}
    ((p^{e-1}-1)/3,p^{e-1}-1)               & \text{if $p=2$,} \\
    (p^{e-1}-1,(p^{e}-3)/2)               & \text{if $p=3$,} \\
    ((p^e-1)/3,(p^e-1)/2)                    & \text{if $p \equiv 5 \mod 6$.}
  \end{cases}
\end{equation*}
\end{claimproof}

We take the minimum $n$ such that the coefficient of 
\[
(x^3)^{A_0}(y^2z)^{B_0}
\]
in $f^{t_ep^e}$ is not contained in $(p^n)$, where $A_0,B_0$ are in the proof of Claim \ref{claim: cusp 1}.

\begin{claim}[p=2]\label{cusp:char=2}
If $p=2$, then $n=e-1$.
Furthermore, we define $a \in A$ by
\[
a:=(x^{p^{e-1}(p-1)}yz^{p^{e-1}})^{p^{n-1}}(xyz)^{p^{n-1}-1}=x^{p^{e+n-2}(p-1)+p^{n-1}-1}y^{p^{n}-1}z^{p^{e+n-2}+p^{n-1}-1},
\]
and $g:=f^{t_ep^{e+n-1}}a$.
Then we have $u^{e+r-1}(\phi^{e+r-1}_*g) \in (p^r)$ for $1 \leq r \leq n-1$ and $u^{e+n-1}(\phi^{e+n-1}_*g) \notin (\m,p^n)$.
\end{claim}

\begin{claimproof}
$n=e-1$ follows from the computation of the binomial.
We fix $1 \leq r \leq n-1$.
We consider the term in $g$ with the orders of $x,y,z$ are $-1$ modulo $p^{e+r-1}$.
It can be denoted by
\[
(x^3)^A(y^2z)^Ba,
\]
where $A,B$ are non-negative integers with $A+B=t_ep^{e+n-1}$ and
\begin{align*}
    3A+p^{e+n-2}(p-1)+p^{n-1}-1 &\equiv -1 \mod p^{e+r-1}, \\
    2B+p^n-1 &\equiv -1 \mod p^{e+r-1}, \\
    B+p^{e+n-2}+p^{n-1}-1 &\equiv -1 \mod p^{e+r-1}.
\end{align*}
In particular, we have $A,B \in (p^{n-1})=(p^{e-2})$.
Let $A':=A/p^{n-1}$, $B':=B/p^{n-1}$.
Then we have $A'+B'=t_ep^e$ and
\begin{align*}
   3A' &\equiv -1 \mod p^{r+1}, \\
    B' &\equiv -1 \mod p^{r}.
\end{align*}
By the computation of the binomial, we have 
\[
\binom{t_ep^e}{A'} \in (p^r).
\]
On the other hand, we consider the case of $r=n=e-1$, then we have $A=p^{n-1}A_0$ and $B=p^{n-1}B_0$.
Therefore, we have $u^{e+n-1}(\phi^{e+n-1}_*g) \notin (\m,p^n)$.
\end{claimproof}

\begin{claim}[p=3]\label{cusp:char=3}
If $p=3$, then $n=e-1$.
Furthermore, we define $a \in A$ by
\[
a:=(x^2y^2z^{(p^{e}+1)/2})^{p^{n-1}}(xyz)^{p^{n-1}-1}
\]
and $g:=f^{t_ep^{e+n-1}}a$.
Then we have $u^{e+r-1}(\phi^{e+r-1}_*g) \in (p^r)$ for $1 \leq r \leq n-1$ and $u^{e+n-1}(\phi^{e+n-1}_*g) \notin (\m,p^n)$.
\end{claim}

\begin{claimproof}
By a similar argument of Claim \ref{cusp:char=2}, we have to consider non-negative integers $A,B$ such that $A+B=t_ep^{e+n-1}$ and
\begin{align*}
    3A+p^n-1 &\equiv -1 \mod p^{e+r-1}, \\
    2B+p^n-1 &\equiv -1 \mod p^{e+r-1}, \\
\end{align*}
thus, $A,B$ are contained in $p^{n-1}$.
We put $A'=A/p^{n-1}$ and $B'=B/p^{n-1}$, then we have $A'+B'=t_ep^e$ and
\begin{align*}
    A' &\equiv -1 \mod p^{r}, \\
    2B'-p^{e-1} &\equiv -1 \mod p^{r+1}. \\
\end{align*}
In particular, if $r \leq n-1$, then we have
\[
\binom{t_ep^e}{B'} \in (p^r).
\]
If $r=n$, then $A'=A_0$ and $B'=B_0$, thus we have $u^{e+n-1}(g) \notin (\m,p^n)$.
\end{claimproof}

\begin{claim}[$p \equiv 5 \mod 6$]\label{cusp:otherwise}
If $p \equiv 5 \mod 6$, then $n=e/2+1$.
Furthermore, we define 
\[
a:=(z^{(p^e-1)/2})^{p^{n-1}}(xyz)^{p^{n-1}-1} \quad \text{and} \quad g:=f^{t_ep^{e+n-1}}a.
\]
Then we have $u^{e+r-1}(\phi^{e+r-1}_*g) \in (p^r)$ for $1 \leq r \leq n-1$ and $u^{e+n-1}(\phi^{e+n-1}_*g) \notin (\m,p^n)$.
\end{claim}

\begin{claimproof}
By a similar argument of Claim \ref{cusp:char=2}, we have to consider non-negative integers $A,B$ such that $A+B=t_ep^{e+n-1}$ and
\begin{align*}
    3A+p^{n-1}-1 &\equiv -1 \mod p^{e+r-1}, \\
    2B+p^{n-1}-1 &\equiv -1 \mod p^{e+r-1}, \\
\end{align*}
thus, $A,B$ are contained in $p^{n-1}$.
We put $A'=A/p^{n-1}$ and $B'=B/p^{n-1}$, then we have $A'+B'=t_ep^e$ and
\begin{align*}
    3A' &\equiv -1 \mod p^{e/2+r-1}, \\
    2B' &\equiv -1 \mod p^{e/2+r-1}. \\
\end{align*}
Therefore we have
\[
\binom{t_ep^e}{B'} \in (p^{\lfloor (n+r-1)/2 \rfloor}) \subseteq (p^r)
\]
if $r \leq n-1$.
Furthermore, if $r=n$, then $A'=A_0$ and $B'=B_0$, thus we have $u^{e+n-1}(g) \notin (\m,p^n)$.
\end{claimproof}

Therefore, by Theorem \ref{thm:log criterion}, $(X,t_eD)$ is quasi-$F^e$-split.
In particular, we have $\qfpt(X;D) \geq t_e$.
Since the limit of $t_e$ is $5/6$, we have $\qfpt(X;D) \geq 5/6$. 
\end{example}

\bibliographystyle{skalpha}
\bibliography{bibliography.bib}

\end{document}